\documentclass[11pt]{amsart}
\usepackage{amsmath, amsthm, amssymb}
\usepackage[margin=1in]{geometry}
\usepackage{epsfig}
\usepackage{rawfonts}
\usepackage{enumerate}
\usepackage{graphics}
\usepackage{multirow}
\usepackage{xspace}
\usepackage{graphicx}
\usepackage{pgf,tikz,pgfplots}
\usepackage{mathrsfs}
\usetikzlibrary{arrows}
\usepackage{amsmath}
\usepackage{amsfonts}
\usepackage{amssymb}
\usepackage{amsthm}
\usepackage{graphicx}
\usepackage{booktabs}
\usepackage{caption}
\usepackage{listings}
\usepackage{setspace}
\usepackage[mathscr]{eucal}
\usepackage{pgfplots}
\usepackage{hyperref}
\usepackage{wrapfig}
\usepackage{floatflt,epsfig}
\usepackage{dsfont}
\usepackage{amscd}
\usepackage{tikz-cd}
\usepackage{fancyhdr}
\usepackage[all]{xy}
\usepackage{latexsym}
\usepackage{amscd}
\usepackage{pifont}
\usepackage{eufrak}
\usepackage{subfig}
\usepackage{easyReview}

\newcommand{\numberset}{\mathbb}
\newcommand{\N}{\numberset{N}}

\newcommand{\kthick}[1]{\textup{T}_{#1}}

\newcommand{\cC}{\mathcal{C}}

\newcommand{\cH}{\mathcal{H}}

\newcommand{\cK}{\mathcal{K}}
\newcommand{\cB}{\mathcal{B}}

\theoremstyle{plain}

\theoremstyle{theorem}

\newtheorem{definition}{Definition}[section]
\newtheorem{proposition}[definition]{Proposition}
\newtheorem{theorem}[definition]{Theorem}
\newtheorem{lemma}[definition]{Lemma}

\newtheorem{corollary}[definition]{Corollary}
\newtheorem{example}[definition]{Example}
\newtheorem{question}[definition]{Question}
\newtheorem{remark}[definition]{Remark}

\theoremstyle{remark}

\makeatletter
\@namedef{subjclassname@2020}{\textup{2020} Mathematics Subject Classification}
\makeatother

\begin{document}
		
		\title[On unboundedness of some invariants of $\mathcal{C}$-semigroups]{On unboundedness of some invariants of $\mathcal{C}$-semigroups}

         \author{OM PRAKASH BHARDWAJ}
		\address{Department of Mathematics, Indian Institute of Technology Bombay, Powai, Maharashtra 400076, India
			}
		\email{om.prakash@math.iitb.ac.in}        
  
		\author{CARMELO CISTO}
		\address{Universit\`{a} di Messina, Dipartimento di Scienze Matematiche e Informatiche, Scienze Fisiche e Scienze della Terra\\
			Viale Ferdinando Stagno D'Alcontres 31\\
			98166 Messina, Italy}
		\email{carmelo.cisto@unime.it}

      \keywords{$\cC$-semigroup, conductor, Ap\'ery set, type, reduced type}
		
		\subjclass[2020]{20M14, 13D02, 11D07, 05E40}

		\begin{abstract}

In this article, we first prove that the type of an affine semigroup ring is equal to the number of maximal elements of the Ap\'ery set with respect to the set of exponents of the monomials, which form a maximal regular sequence. Further, we consider $\mathcal{C}$-semigroups in $\mathbb{N}^d$ and prove that the notions of symmetric and almost symmetric $\mathcal{C}$-semigroups are independent of term orders. We further investigate the conductor and the Ap\'ery set of a $\mathcal{C}$-semigroup with respect to a minimal extremal ray. Building upon this, we extend the notion of reduced type to $\mathcal{C}$-semigroups and study its extremal behavior. For all $d$ and fixed $e \geq 2d$, we give a class of $\mathcal{C}$-semigroups of embedding dimension $e$ such that both the type and the reduced type do not have any upper bound in terms of the embedding dimension. We further explore irreducible decompositions of a $\mathcal{C}$-semigroup and give a lower bound on the irreducible components in an irreducible decomposition. Consequently, we deduce that for each positive integer $k$, there exists a $\mathcal{C}$-semigroup $S$ such that the number of irreducible components of $S$ is at least $k$.

		\end{abstract}

		\maketitle

\section{Introduction}

Let $\mathbb{Z}$ be the set of integers, $\mathbb{N}$ be the set of nonnegative integers, and $\mathbb{Q}_{\geq 0}$ be the set of nonnegative rational numbers. In this article, the main objects of study are affine $\mathcal{C}$-semigroups. These semigroups are commonly encountered in the literature. Still, their independent study began with the article \cite{garcia2018extension}, where the authors extended the well-known Wilf's conjecture for numerical semigroups to the affine semigroups. Recall that an affine semigroup is a finitely generated submonoid of $\mathbb{N}^d$ for some positive integer $d$. It is known that an affine semigroup has a unique minimal set of generators, and the cardinality of the minimal generating set is known as the embedding dimension of $S$, denoted by $\operatorname{e}(S)$. Let $ S \subseteq \mathbb{N}^d $ be an affine semigroup. The set $ \mathrm{cone}(S) = \mathrm{Span}_{\mathbb{Q}_{\geq 0}}(S) $ denotes the rational cone generated by $ S $. The set $ \mathcal{H}(S) = (\mathrm{cone}(S) \setminus S) \cap \mathbb{N}^d $ is referred to as the set of \textit{gaps} of $ S $, and an element $\mathbf{f} \in \mathcal{H}(S)$ satisfying $\mathbf{f} + S \setminus \{0\} \subseteq S$ referred to as a pseudo-Frobenius element of $S$. Define $ \mathcal{C} = \mathrm{cone}(S) \cap \mathbb{N}^d $; this is an affine semigroup and known as the \textit{integral closure} of $ S $. An affine semigroup is called a \textit{$ \mathcal{C}$-semigroup} if $ \mathcal{H}(S) $ is finite. When $d=1$ and $ S $ is a $ \mathcal{C} $-semigroup, it is known as a \textit{numerical semigroup}. Numerical semigroups have rich literature and connections to combinatorics, commutative algebra, number theory, and algebraic geometry. If $ \mathcal{C} = \mathbb{N}^d $ (referred to as the full cone), then $ S $ is known as a \textit{generalized numerical semigroup}, introduced in \cite{failla2016algorithms}.

In this article, we aim to explore some invariants of $\mathcal{C}$-semigroups. These invariants have nice algebraic and homological correspondence to the invariants of the corresponding semigroup ring. Let us first recall the definition of a semigroup ring. Let $K$ be a field. For an affine semigroup $S \subseteq \mathbb{N}^d$, the semigroup ring $K[S]$ is defined as the $K$-subalgebra of the polynomial ring $K[t_1, \ldots, t_d]$ generated by the monomials $\mathbf{t}^{\mathbf{s}}$; where $\mathbf{s} \in S.$ In \cite{garcia-pseudofrobenius}, the authors establish that an affine semigroup $ S $ possesses pseudo-Frobenius elements if and only if the length of the minimal graded free resolution of its semigroup ring $ K[S] $ is maximally possible. Since $ K[S] $ is an integral domain, according to the Auslander-Buchsbaum formula, this condition translates to $ S $ having pseudo-Frobenius elements precisely when $ \mathrm{depth}(K[S]) = 1 $. Given that a $ \mathcal{C} $-semigroup always contains pseudo-Frobenius elements, its semigroup ring invariably maintains depth one. Thus, such semigroups are never Cohen-Macaulay unless $ d = 1 $. Moreover, in accordance with \cite[Corollary 7]{garcia-pseudofrobenius}, the cardinality of the set of pseudo-Frobenius elements coincides with the last Betti number of the corresponding semigroup ring, which further equals the type of the semigroup ring. Additionally, the semigroup rings associated with certain generalized numerical semigroups in $ \mathbb{N}^2 $ have previously appeared in literature as the associated graded rings of the ring of differential operators of affine monomial curves (see \cite{eriksen}). The non-Cohen-Macaulay nature of $ \mathcal{C} $-semigroups thus provides a complete answer regarding the Cohen-Macaulay property of the associated graded rings of the ring of differential operators of affine monomial curves. The invariants and properties of $\mathcal{C}$-semigroups, which we are mainly interested in, are related to the \emph{type} of the corresponding semigroup ring, which is a classical invariant defined for Noetherian commutative rings (see, for instance, \cite{bruns1998cohen}). We refer to the \emph{type} of the semigroup $S$ as the type of the semigroup ring $K[S]$. In the context of numerical semigroups, it is known that the type is equal to the number of the pseudo-Frobenius elements of the semigroups. Some interesting studies about the properties of the type of a numerical semigroup have been performed. One of these concerns with the investigation of the behavior of the type in particular families of numerical semigroups having fixed embedding dimensions, to find if the type is bounded or not in these families, see for instance \cite{moscarielloedim4, moscarielloedim5}. For another topic, a new invariant called \emph{reduced type} has been recently introduced in \cite{huneke2021torsion}, which can be seen as a lower bound of the type. Some properties of the reduced type for numerical semigroups have been studied in \cite{maitra2023extremal}. The aim of our work is to extend these studies in the context of $\mathcal{C}$-semigroups. In particular, we provide some results in this more general context and extend some known properties of numerical semigroups in this framework.

 Now, we summarize the contents of the paper. In Section 2, we recall some basic definitions and notions which will be used throughout the article. In Section 3, we give the definition of the type for an affine semigroup and relate it to the cardinality of the maximal elements of the Ap\'ery set of a specific subset in Theorem~\ref{thm:type_apery}. In Section 4, we study some important properties of $\cC$-semigroups. We start by considering \cite[Theorem 2]{singhal} in a more general context, and as a consequence in Corollary~\ref{cor:FA(S)}, we prove that in $\mathcal{C}$-semigroups the set of Frobenius allowable elements is equal to the set of maximal gaps with respect to the usual partial order on $\mathbb{N}^d$. We next consider symmetric and almost symmetric $\mathcal{C}$-semigroups; we show that these notions are independent of term orders and give characterizations in Propostions~\ref{prop:C-symmeric} and \ref{prop:C-almost-symmetric}, which solely depend on the semigroups. We also extend the concept of quasi-irreducible (given in \cite{singhal} ) to $\mathcal{C}$-semigroup by giving equivalent criteria in Theorem~\ref{thm:equivalent-quasiIrred}. In Section 5, we introduce the concept of reduced type. In order to study this invariant, we give the description of conductor for $\mathcal{C}$-semigroups in Proposition~\ref{prop:conductor}. In Propositions~\ref{prop:maximal-redtype} and \ref{equivalent_minreltype}, we provide characterizations of $\mathcal{C}$-semigroups to have minimal or maximal reduced type. These properties are also investigated for some classes of generalized numerical semigroups, for instance, for $T$-graded generalized numerical semigroups in Theorem~\ref{thm:tgraded_redtype} and for semigroups obtained by the so-called thickening in Theorem~\ref{thm:thick_redtype}. In Section 5, we study the unboundedness properties of the invariants type, reduced type, and minimal number of irreducible components of a $\mathcal{C}$-semigroup. In particular, by combining Remark~\ref{remarktyped=1}, Proposition~\ref{prop:unbound_typeGNS}, we conclude that for all values $e\geq 2d$, there exists a family of $\mathcal{C}$-semigroups having fixed embedding dimension equal to $e$ and unbounded type. We find the same result for the reduced type by considering Proposition~\ref{redtypeextension} and ~\ref{unboundedredtype}. Further, we consider the decomposition of a $\mathcal{C}$-semigroups into irreducible $\mathcal{C}$-semigroups, and in Theorem~\ref{thm:unboundeness-irreducible}, we give a lower bound on the number of components in an irreducible decomposition. Consequently, in Corollary~\ref{unbooundeddecomposition}, we prove that for every positive integer $k$, there exist $\mathcal{C}$-semigroups having the number of irreducible components greater than or equal to $k$. 

\section{Preliminaries}
In this section, we recall the main notions used in the paper. Let $S\subseteq \mathbb{N}^d$ be an affine semigroup. We denote $\mathrm{cone}(S)=\mathrm{Span}_{\mathbb{Q}_{\geq 0}}(S)$. Let $A$ be the set of minimal generators of $S$, then there exists a subset $B\subseteq A$ such that $\mathrm{cone}(S)=\mathrm{Span}_{\mathbb{Q}_{\geq 0}}(B)$ and for all $B'\subsetneq B$ then $\mathrm{cone}(S)\neq \mathrm{Span}_{\mathbb{Q}_{\geq 0}}(B')$. A set having the property of $B$ is called a set of \emph{extremal rays} of $S$. Assume $B=\{\mathbf{b}_1,\ldots,\mathbf{b}_t\}$ for some $t\in \mathbb{N}$. For all $i\in \{1,\ldots,t\}$, let us denote $\mathbf{n}_i=\min \{q\mathbf{b}_i \in S\mid q\in \mathbb{Q}_{\geq 0}\}$. Define $E=\{\mathbf{n}_1,\ldots,\mathbf{n}_t\}$. Obviously, $E$ is a set of extremal rays of $S$ that we call \textit{minimal extremal rays}. If $t=d$, then $S$ is called a \textit{simplicial} affine semigroup. 

\medskip
Let $X$ be a non-empty subset of $S$, define
$$\mathrm{Ap}(S,X)=\{\mathbf{s}\in S\mid \mathbf{s}-\mathbf{x}\notin S\ \text{ for all } \mathbf{x}\in X\}.$$
The set $\mathrm{Ap}(S,X)$ is called the \textit{Ap\'ery set} of $S$ with respect to $X$. If $\mathbf{x}\in S$, simplifying the notation, for $\mathrm{Ap}(S,\mathbf{x})$ we mean $\mathrm{Ap}(S,\{\mathbf{x}\})$. The set $\mathrm{Ap}(S,X)$ can be infinite, but it is known that $\mathrm{Ap}(S, E)$ is a finite set (see, for instance, \cite{jafari2024depth}).


\medskip
For a subset $L \subseteq \mathbb{N}^d$, define the partial order $\leq_L$ on $\mathbb{N}^d$ as follows: $\mathbf{x} \leq_L \mathbf{y}$ if and only if $\mathbf{y}-\mathbf{x} \in L.$ The set of gaps of $S$ is defined as $\cH(S):=(\mathrm{cone}(S)\setminus S)\cap \mathbb{N}^d$.  An element $\mathbf{f}\in \cH(S)$ is called a\textit{ pseudo-Frobenius} element if $\mathbf{f} + S\setminus\{0\} \subseteq S$. The set of pseudo-Frobenius elements is denoted by $\mathrm{PF}(S)$. In particular,
$$\operatorname{PF}(S)=\{\mathbf{h}\in \cH(S)\mid \mathbf{h}+\mathbf{s}\in S, \text{ for all }\mathbf{s}\in S\setminus \{\mathbf{0}\}\}.$$
An affine semigroup $S$ such that $\operatorname{PF}(S)\neq \emptyset$ is called a \textit{maximal projective dimension} (in short MPD) semigroup (see \cite{garcia-pseudofrobenius}). In this kind of semigroup, there is a relation between pseudo-Frobenius elements and the Ap\'ery set with respect to some element in the semigroup. Being careful with the \cite[Definition 14]{garcia-pseudofrobenius}, and using the same arguments as of  \cite[Proposition 16]{garcia-pseudofrobenius}, it can be verified that $\operatorname{PF}(S)=\{\mathbf{y}-\mathbf{x}\mid \mathbf{y}\in \mathrm{Maximals}_{\leq_S}\mathrm{Ap}(S,\mathbf{x})\} $. This fact allows us to state the following result.
\begin{proposition}\label{prop:aperyPF}
Let $S$ be an MPD-semigroup and $\mathbf{x}\in S$. Then
$$\mathrm{Maximals}_{\leq_S}\mathrm{Ap}(S,\mathbf{x})=\{\mathbf{h}+\mathbf{x}\mid\mathbf{h}\in \operatorname{PF}(S)\}.$$
\end{proposition}

\medskip
When the complement of an affine semigroup $S$ in its integer cone $\cC=\mathrm{cone}(S)\cap \mathbb{N}^d$ is finite, it is called $\mathcal{C}$-semigroup. Note that $\mathcal{C}$-semigroups are a particular case of MPD-semigroups. In fact, if $S$ is a $\mathcal{C}$-semigroup, then $\cH(S)$ is finite, and the maximal gaps with respect to the order $\leq_S$ are pseudo-Frobenius elements.

\medskip
Now, we recall the definition of \textit{conductor} of an affine semigroup $S$. Denote $\mathrm{group}(S)=\{\mathbf{a}-\mathbf{b}\in \mathbb{Z}^d\mid \mathbf{a},\mathbf{b}\in S\}$ and consider the set 
 $\overline{S}=\mathrm{cone}(S)\cap \mathrm{group}(S)$, which is a normal affine semigroup and it is called the \textit{normalization} of $S$. The \textit{conductor} of $S$ is the set $\mathfrak{C}_S=\{\mathbf{s}\in S\mid \mathbf{s}+\overline{S}\subseteq S\}$.

\medskip
 Let $S\subseteq \mathbb{N}^d$ be a $\mathcal{C}$-semigroup. It is known that in this case $\overline{S}=\mathcal{C}$. In this article, when $\mathcal{C}=\mathbb{N}^d$, we call $S$ a \textit{generalized nuemrical semigroup}. While, in the case $d=1$, $S$ is called a \textit{numerical semigroup}. Numerical semigroups are largely studied in the literature (see, for instance, \cite{rosales2009numerical}). If $S$ is a numerical semigroup, we remind that the element $\operatorname{F}(S)=\max (\mathbb{Z}\setminus S)$ is called the \textit{Frobenius number} of $S$ and the number $\operatorname{m}(S)=\min S\setminus\{0\}$ is called the \textit{multiplicity} of $S$. Moreover, in this case, $\mathfrak{C}_S=\{n\in \mathbb{N}\mid n> \operatorname{F}(S)\}$.

\section{Semigroup ring and type of an affine semigroup}
Assume $R$ is a Noetherian graded ring with maximal homogeneous ideal $\mathfrak{m}$ (or a Noetherian local ring with maximal ideal $\mathfrak{m}$). The \emph{type} of $R$  is the value $\mathrm{t}(R)=\dim_K\left(\frac{\mathfrak{n}R:\mathfrak{m}}{\mathfrak{n}R}\right)$ where $\mathfrak{n}$ is any ideal generated by a maximal regular sequence on $R$ (see \cite[Lemma 1.2.19]{bruns1998cohen}).

\medskip
 Let $S=\langle \mathbf{n}_1,\ldots,\ldots,\mathbf{n}_{r}\rangle \subseteq \mathbb{N}^d$ be an affine semigroup minimally generated by $\mathbf{n}_1,\ldots,\ldots,\mathbf{n}_{r}$. We will follow this notation throughout the article for an affine semigroup minimally generated by $\mathbf{n}_1,\ldots,\ldots,\mathbf{n}_{r}$. Let $K$ be a field and $S = \langle \mathbf{n}_1,\ldots,\ldots,\mathbf{n}_{r}\rangle$ be an affine semigroup, the \textit{semigroup ring} $K[S]:=K[X^\mathbf{s}\mid \mathbf{s}\in S]$ is a $K$-subalgebra of the polynomial ring $K[X_1, \ldots, X_d]$. The ring $K[S]$ is Noetherian $S$-graded with maximal homogeneous ideal $\mathfrak{m}=(X^{\mathbf{n}_1},\ldots,X^{\mathbf{n}_{r}})$. We denote $\operatorname{t}(S)=\operatorname{t}(K[S])$, and we call it the \textit{type} of $S$. In \cite[Section 3]{jafari2022type}, it is showed that $\operatorname{t}(K[S])=\operatorname{t}(K[S]_{\mathfrak{m}})$. 
The following relation between $\operatorname{t}(S)$ and some Ap\'ery set of $S$ with respect to the exponents of the monomials, which form a maximal regular sequence in $K[S].$

\begin{theorem}\label{thm:type_apery}
Let $S=\langle \mathbf{n}_1,\ldots,\ldots,\mathbf{n}_{r}\rangle\subseteq \mathbb{N}^d$ be an affine semigroup and $R=K[S]$. Assume $Q=\{\mathbf{a}_1,\ldots,\mathbf{a}_m\}$ such that $\{X^{\mathbf{a}_1},\ldots,X^{\mathbf{a}_m}\}$ is a maximal regular sequence in $R$. Then $\operatorname{t}(S)= |\mathrm{Maximals}_{\leq_S}\mathrm{Ap}(S,Q)|.$
\end{theorem}
\begin{proof}
 Denote $\mathfrak{n}=(X^{\mathbf{a}_1},\ldots,X^{\mathbf{a}_m})$. We will prove that 
$$\frac{\mathfrak{n}R:\mathfrak{m}}{\mathfrak{n}R}=\mathrm{Span}_K\left (\{X^{\mathbf{t}}+\mathfrak{n}\mid \mathbf{t}\in \mathrm{Maximals}_{\leq_S}\mathrm{Ap}(S,Q)\}\right).$$
After this, the result will follow from the hypothesis that $\{X^{\mathbf{a}_1},\ldots,X^{\mathbf{a}_m}\}$ is a maximal regular sequence in $R$. First observe that if $f+\mathfrak{n}\in R/\mathfrak{n}$, then $f+\mathfrak{n}=\sum_{i\in I}k_iX^{\mathbf{s}_i}+\mathfrak{n}$, for some finite set $I$, $k_i\in K$ and $\mathbf{s}_i\in S$ with $X^{\mathbf{s}_i}\notin \mathfrak{n}$ for all $i\in I$. In particular, we have $\mathbf{s}_i\in \mathrm{Ap}(S,Q)$ for all $i\in I$. In fact, if $\mathbf{s}_i\notin \mathrm{Ap}(S,Q)$, then there exists $\mathbf{a}_j\in Q$ such that $\mathbf{s}_i=\mathbf{s}+\mathbf{a}_j$ for some $\mathbf{s}\in S$. This means $X^{\mathbf{s}_i}=X^{\mathbf{s}}X^{\mathbf{a}_j}\in \mathfrak{n}$, a contradiction. So, $\mathbf{s}_i\in \mathrm{Ap}(S,Q)$ for all $i\in I$. If $f+\mathfrak{n}=\sum_{i\in I}k_iX^{\mathbf{s}_i}+\mathfrak{n} \in \frac{\mathfrak{n}R:\mathfrak{m}}{\mathfrak{n}R}$, then $X^{\mathbf{s}_i+\mathbf{n}_j}\in \mathfrak{n}$ for all $i\in I$ and for all $j\in \{1,\ldots,r\}$. In particular, $\mathbf{s}_i+\mathbf{n}_j-\mathbf{a}_\ell\in S$ for some $\ell\in \{1,\ldots,m\}$. This implies that $\mathbf{s}_i+\mathbf{s}\notin \mathrm{Ap}(S,Q)$ for all $\mathbf{s}\in S\setminus \{\mathbf{0\}}$ and this means $\mathbf{s}_i\in \mathrm{Maximals}_{\leq_S}\mathrm{Ap}(S,Q)$, for all $i\in I$.\newline
\noindent Conversely, let $\overline{f}=f+\mathfrak{n}=\sum_{i\in I}k_iX^{\mathbf{s}_i}+\mathfrak{n}$ for some finite set $I$, such that $\mathbf{s}_i\in \mathrm{Maximals}_{\leq_S}\mathrm{Ap}(S,Q)$, for all $i\in I$. We need to prove that $\overline{f}\mathfrak{m}\in \mathfrak{n}$. In particular, it is sufficient to prove that $\overline{f}X^{\mathbf{n}_j}\in \mathfrak{n}$ for all $j\in \{1,\ldots,r\}$. If $j\in \{1,\ldots,r\}$, then $\overline{f}X^{\mathbf{n}_j}=\sum_{i\in I}k_iX^{\mathbf{s}_i+\mathbf{n}_j}+\mathfrak{n}$. Since $\mathbf{s}_i\in \mathrm{Maximals}_{\leq_S}\mathrm{Ap}(S,Q)$ for all $i\in I$, we have $\mathbf{s}_i+\mathbf{n}_j\notin \mathrm{Ap}(S,Q)$ for all $i\in I$. This implies $\mathbf{s}_i+\mathbf{n}_j=\mathbf{s}+\mathbf{a}_\ell$ for some $\mathbf{s}\in S$ and $\mathbf{a}_\ell \in Q$. In particular, $X^{\mathbf{s}_i+\mathbf{n}_j}=X^{\mathbf{s}}\cdot X^{\mathbf{a}_\ell}\in \mathfrak{n}$, obtaining $\overline{f}X^{\mathbf{n}_j}\in \mathfrak{n}$.  
\end{proof}

From the above theorem, we get a nice characterization of the type of a semigroup ring associated with an MPD-semigroup in terms of the Ap\'ery set and pseudo-Frobenius elements. The commutative algebraic equivalent definition of an MPD-semigroup is the following.

\begin{definition}
{\rm An affine semigroup $S=\langle \mathbf{n}_1,\ldots,\ldots,\mathbf{n}_{r}\rangle$ is called maximal projective dimension (MPD) semigroup if the projective dimension $\mathrm{pdim}(K[S]) = r-1.$ Equivalently, $\mathrm{depth}(K[S])=1.$ }   
\end{definition}

\begin{corollary}
    Let $S=\langle \mathbf{n}_1,\ldots,\ldots,\mathbf{n}_{r}\rangle$ be an MPD-semigroup and $\bf{x}$ be a non-zero element of $S$. Then
    $$\operatorname{t}(K[S]) = \mid \mathrm{Maximals}_{\leq_S}\mathrm{Ap}(S,\mathbf{x}) \mid = \mid \mathrm{PF}(S) \mid =   \beta_{r-1}(K[S]),$$
    where $\beta_{r-1}(K[S])$ is the last Betti number of the semigroup ring $K[S].$
\end{corollary}
\begin{proof}
    Since $S$ is an MPD-semigroup, we have $\mathrm{depth}(K[S])=1$. Since $K[S]$ is a domain, for any non-zero element $\mathbf{n}\in S$, we have $\{X^\mathbf{n}\}$ is a maximal regular sequence in $K[S]$. Now, the first equality follows from the Theorem~\ref{thm:type_apery}. Second equality follows from the Proposition~\ref{prop:aperyPF}. The last equality follows from \cite[Corollary 7]{garcia-pseudofrobenius},
\end{proof}

\begin{remark}
    { \rm Theorem~\ref{thm:type_apery} generalizes \cite[Proposition 3.3]{jafari2022type}, where the authors prove that the type of a Cohen-Macaulay simplicial affine semigroup is equal to the cardinality of the Ap\'ery set with respect to its extremal rays. If $S=\langle \mathbf{n}_1,\ldots,\mathbf{n}_d, \mathbf{n}_{t+1},\ldots,\mathbf{n}_{d+r}\rangle \subseteq \mathbb{N}^d$ is a simplicial affine semigroup with minimal extremal rays $E=\{\mathbf{n}_1,\ldots,\mathbf{n}_d\}$ then $\{X^{\mathbf{n}_1},\ldots,X^{\mathbf{n}_d}\}$ form a system of parameters for $K[S]$. Since $K[S]$ is a Cohen-Macaulay, then the set $\{X^{\mathbf{n}_1},\ldots,X^{\mathbf{n}_d}\}$ is a maximal regular sequence of $K[S]$ and $\operatorname{t}(S)=|\mathrm{Maximals}_{\leq_S}\mathrm{Ap}(S,E)|$.}
\end{remark}

\begin{remark}
   {\rm For a commutative ring $R$, the \emph{conductor} ideal is defined as $\mathfrak{C}_R=(R:_\mathcal{Q} \overline{R})$, where $\mathcal{Q}$ is the total ring of fractions of $R$ and $\overline{R}$ is the normalization of $R$. In the case when $R=K[S]$ is a semigroup ring, the conductor of $R$ is related to notion of conductor of $S$ introduced before, since $\mathfrak{C}_{K[S]}=(X^\mathbf{c}\mid \mathbf{c}\in \mathfrak{C}_S)$.}
\end{remark}

\section{Some properties of $\cC$-semigroups}
In this section, we study some important properties of $\cC$-semigroups. We start by considering \cite[Theorem 2]{singhal} in a more general context. Recall that a total order $\preceq$ in $\mathbb{N}^d$ is a \textit{relaxed monomial order} if: 1) for all $\mathbf{u},\mathbf{v}$ such that $\mathbf{u}\preceq \mathbf{v}$, then $\mathbf{u}\preceq \mathbf{v}+\mathbf{w}$ for all $\mathbf{w}\in \mathbb{N}^d$; 2) for all $\mathbf{v}\in \mathbb{N}^d$, then $\mathbf{0}\preceq \mathbf{v}$.

\begin{theorem}\label{thm:antichain}
    Let $A\subseteq \mathbb{N}^{d}$ be an antichain with  respect to the order $\leq_{\mathbb{N}^d}$. Define the set $$\mathcal{B}(A)=\{\mathbf{x}\in \mathbb{N}^d\setminus \{\mathbf{0}\} \mid \mathbf{x}\leq_{\mathbb{N}^d} \mathbf{a}\text{ for some }\mathbf{a}\in A\},$$ then for all $\mathbf{a} \in A$ there exists a relaxed monomial order $\preceq$ such that $\mathbf{a}=\max_{\preceq}\mathcal{B}(A)$. 
\end{theorem}
\begin{proof}
   Let us denote $S=\mathbb{N}^d \setminus \mathcal{B}(A)$. Observe that $S$ is trivially a generalized numerical semigroup such that $\mathrm{Maximals}_{\leq_{\mathbb{N}^d}}\mathcal{H}(S)=A$. Now, the proof is essentially the same as the proof of \cite[Theorem 2]{singhal}.
\end{proof}

 If $S$ is a $\mathcal{C}$-semigroup and $\preceq$ is a relaxed monomial order, for the element $\mathbf{F}_\preceq (S)$ we mean the maximum in the set $\mathcal{H}(S)$ with respect to $\preceq$. Let us denote by $\operatorname{FA}(S)=\{\mathbf{x}\in \mathcal{H}(S)\mid \mathbf{x}=\mathbf{F}_{\preceq}(S)\text{ for some relaxed monomial order }\preceq\}$, called the set of \textit{Frobenius allowable} elements of $S$. The set $\operatorname{SG}(S)=\{\mathbf{h}\in \operatorname{PF}(S)\mid 2\mathbf{h}\in S\}$, called the set of \textit{special gaps} of $S$.

\begin{corollary}\label{cor:FA(S)}
    Let $S$ be a $\mathcal{C}$-semigroup. Then $\operatorname{FA}(S)=\mathrm{Maximals}_{\leq_{\mathbb{N}^d}}\mathcal{H}(S)$.
\end{corollary}
\begin{proof}
 $\subseteq)$ Let $\mathbf{x}\in \operatorname{FA}(S)$. If there exists $\mathbf{y} \in \mathcal{H}(S)\setminus \{\mathbf{x}\}$ such that $\mathbf{x}\leq_{\mathbb{N}^d} \mathbf{y}$, then $\mathbf{x}\prec \mathbf{y}$ for all relaxed monomial order $\prec$. But this means $\mathbf{x}\notin \operatorname{FA}(S)$, so $\mathbf{x}\in \mathrm{Maximals}_{\leq_{\mathbb{N}^d}}\mathcal{H}(S)$. 

 \noindent$ \supseteq)$ Let us denote $A=\mathrm{Maximals}_{\leq_{\mathbb{N}^d}}\mathcal{H}(S)$. Observe that $A$ is an antichain and $\mathcal{H}(S)\subseteq \mathcal{B}(A)$. Let $\mathbf{a}\in A$, by Theorem~\ref{thm:antichain} we know that there exists a relaxed monomial order $\preceq$ such that $\mathbf{a}=\max_\preceq \mathcal{B}(A)$. In particular, for all $\mathbf{h}\in \mathcal{H}(S)$ we have $\mathbf{h}\preceq \mathbf{a}$. So $\mathbf{a}=\mathbf{F}_\preceq (S)$, that is, $\mathbf{a}\in \operatorname{FA}(S)$. 
\end{proof}

\begin{proposition}\label{prop:containement}
    For a $\mathcal{C}$-semigroup $S$, we have the following:
    $$ \operatorname{FA}(S) \subseteq \mathrm{Maximals}_{\leq_{\mathcal{C}}}\mathcal{H}(S) \subseteq \operatorname{SG}(S) \subseteq \operatorname{PF}(S).$$
\end{proposition}
\begin{proof}
    The first inclusion follows from Corollary~\ref{cor:FA(S)}, and the last inclusion is obvious. We prove the middle one here.  Let $\mathbf{x}\in \mathrm{Maximals}_{\leq_{\mathcal{C}}}\mathcal{H}(S)$. If $\mathbf{s}\in S\setminus \{\mathbf{0}\}$ then $\mathbf{x}+\mathbf{s}\in \mathcal{C}$ and $\mathbf{x}\leq_{\mathcal{C}} \mathbf{x}+\mathbf{s}$. By maximality property of $\mathbf{x}$, we have $\mathbf{x}+\mathbf{s}\notin \mathcal{H}(S)$. So, $\mathbf{x}+\mathbf{s}\in S$. Moreover, we have also $2\mathbf{x}\in \mathcal{C}$ and $\mathbf{x}\leq_{\mathcal{C}} 2\mathbf{x}$. Hence, by the same argument as before, $2\mathbf{x}\in S$. Therefore, we have $\mathbf{x}\in \operatorname{SG}(S)$. 
\end{proof}

\begin{example} \label{exa:maximals}
\rm
We give some examples of different possible inclusions stated in Proposition~\ref{prop:containement}.
\begin{enumerate}
\item[(a)] Let $S=\langle A \rangle\subseteq \mathbb{N}^3$, where $A$ is the following set:   
\begin{align*}   
A= \left\lbrace
\begin{array}{c}
    (2, 0, 0), (4, 2, 4), (0, 1, 0), (3, 0, 0), (6, 3, 6), (3, 1, 1), \\ [1mm]
    (4, 1, 1), (3, 1, 2), (1, 1, 0), (3, 2, 3), (1, 2, 1) 
\end{array}
\right\rbrace .
\end{align*} 
Actually, $S$ is a $\mathcal{C}$-semigroup, where $\mathcal{C}=\mathrm{Span}_{\mathbb{Q}_+}(\{(1, 0, 0), (2, 1, 2), (0, 1, 0)\})\cap \mathbb{N}^3$. In particular, the set of gaps of $S$ is the following:
\begin{align*}  
\cH(S)= \left\lbrace 
\begin{array}{c}
(1, 0, 0), (1, 1, 1), (2, 1, 1), (2, 1, 2), (2, 2, 1), (2, 2, 2), (2, 3, 2),\\ [1mm] (4, 1, 2),(4, 2, 3), (5, 2, 4), (5, 3, 5), (8, 4, 7) 
\end{array}
 \right\rbrace.
\end{align*}   
In Figure~\ref{Fig:exampleMaximals(a)}, we have a graphic representation of the semigroup $S$: the cone $\cC$ consists of all integer points inside the colored area, and the marked points represent the gaps. The set of points in $\mathrm{Maximals}_{\leq_{\mathcal{C}}}\mathcal{H}(S)$ are marked in blue. In this case, we have $\operatorname{FA}(S)=\{(8,4,7)\}$, while $$\mathrm{Maximals}_{\leq_{\mathcal{C}}}\mathcal{H}(S) = \operatorname{SG}(S) = \operatorname{PF}(S)=\{( 2, 2, 1 ), ( 2, 3, 2 ), ( 4, 1, 2 ), ( 8, 4, 7 )\} .$$
 
\item[(b)]  Let $\mathcal{C}=\mathrm{Span}_{\mathbb{Q}_+}(\{(1,2), (3,1)\})\cap \mathbb{N}^2$ and 
$$S= \mathcal{C}\setminus\{ ( 1, 1 ), ( 1, 2 ), ( 2, 1 ), ( 2, 2 ), ( 2, 3 ), ( 3, 1 ), ( 3, 2 ), ( 3, 3 ) \}.$$


\noindent It is possible to check that $S$ is really a $\mathcal{C}$-semigroup (see for instance \cite[Proposition 16]{garcia_someproperties}) and it is possible to compute that:
\begin{itemize}
\item $\operatorname{FA}(S)=\{(3,3)\}$.
\item $\mathrm{Maximals}_{\leq_{\mathcal{C}}}\mathcal{H}(S)=\{( 2, 3 ), ( 3, 1 ), ( 3, 2 ), ( 3, 3 )\}.$
\item $\operatorname{SG}(S)=\cH(S)\setminus \{(1,1)\}$ and $\operatorname{PF}(S)=\cH(S)$.
\end{itemize}
Hence we have $ \operatorname{FA}(S) \subsetneq \mathrm{Maximals}_{\leq_{\mathcal{C}}}\mathcal{H}(S) \subsetneq \operatorname{SG}(S) \subsetneq \operatorname{PF}(S).$ A graphic representation of $S$ is provided in Figure~\ref{Fig:exampleMaximals(b)}, where the set $\mathrm{Maximals}_{\leq_{\mathcal{C}}}\mathcal{H}(S)$ consists of the blue points, while the other gaps are marked in black.
 
\noindent Furthermore, if we consider $T=S\cup \{(2,1)\}$, then $T$ is a $\mathcal{C}$-semigroup such that $ \operatorname{FA}(T) \subsetneq \mathrm{Maximals}_{\leq_{\mathcal{C}}}\mathcal{H}(T) \subsetneq \operatorname{SG}(T) = \operatorname{PF}(T).$ 
\end{enumerate} 
\end{example}

\begin{figure}[h]
\centering
\subfloat[The semigroup $S$ in Example~\ref{exa:maximals}(a)]{\includegraphics[scale=0.26]{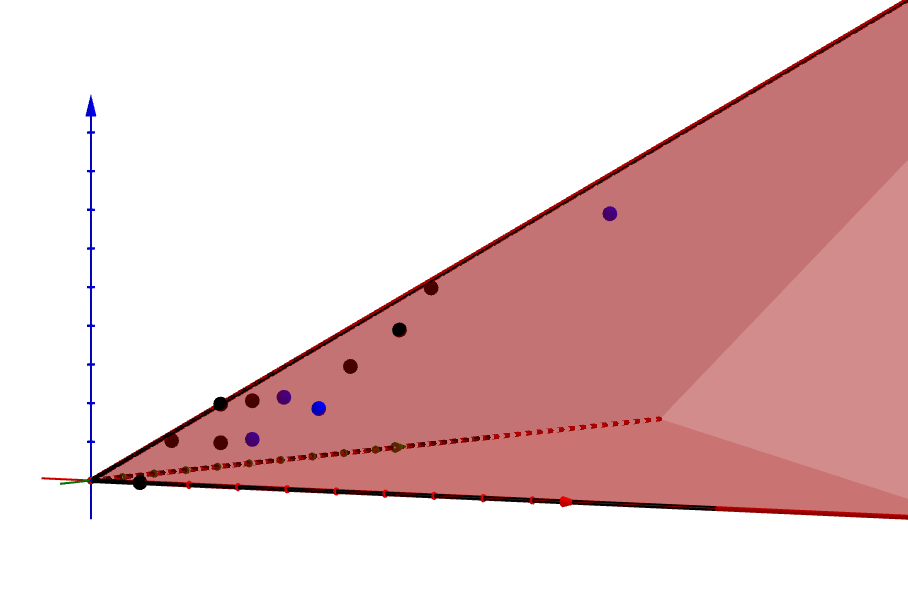} \label{Fig:exampleMaximals(a)}}\qquad \qquad
\subfloat[The semigroup $S$ in Example~\ref{exa:maximals}(b)]{\includegraphics[scale=1.1]{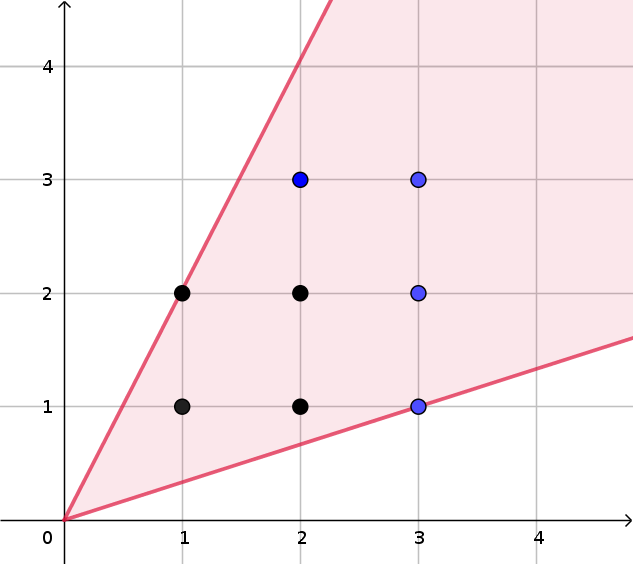}\label{Fig:exampleMaximals(b)}}
\caption{The elements belonging to the semigroups are all integer points inside the colored area, except for the marked points, which represent the gaps.}
\end{figure}

\noindent Recall that a $\mathcal{C}$-semigroup $S$ is called \emph{symmetric} if there exists a term order $\preceq$ such that $\operatorname{PF}(S)=\{\max_{\preceq} \mathcal{H}(S)\}$, while it is called \textit{almost symmetric}  if there exist a term order $\preceq$ such that $\mathbf{F}_\preceq(S) -\mathbf{f} \in \operatorname{PF}(S)$ for all $\mathbf{f}\in \operatorname{PF(S)}\setminus \{\mathbf{F}_\preceq (S)\}$(see for instance \cite{op1, garcia_someproperties}). Note that, with this definition, a symmetric $\mathcal{C}$-semigroup is always almost symmetric. Here and in all the paper, if the maximum of the gapset $\mathcal{H}(S)$ exists with respect to $\leq_{\mathcal{C}}$, we denote it by $\mathbf{F}_{\mathcal{C}}(S)$.

\begin{proposition}\label{prop:C-symmeric}
  Let $S$ be a $\mathcal{C}$-semigroup. Then $S$ is symmetric if and only if $\mathrm{PF}(S) = \{\mathbf{F}_{\mathcal{C}}(S)\}.$   
\end{proposition}
\begin{proof}
Suppose $S$ is symmetric. We first prove that  $\mathbf{F}_{\mathcal{C}}(S)$ exists. Suppose there exist $\mathbf{f}_1,\,\mathbf{f}_2 \in \mathrm{Maximals}_{\leq_\mathcal{C}}(\mathcal{H}(S))$ such that $\mathbf{f}_1\neq \mathbf{f}_2$. By Proposition~\ref{prop:containement} we have $\{\mathbf{f}_1,\mathbf{f}_2\}\subseteq \operatorname{PF}(S)$, a contradiction. Hence, $\mathbf{F}_\mathcal{C}(S)$ exists, and in particular, it is a pseudo-Frobenius element. Therefore, $\mathrm{PF}(S) = \{\mathbf{F}_{\mathcal{C}}(S)\}.$ Conversely, assume $\mathrm{PF}(S) = \{\mathbf{F}_{\mathcal{C}}(S)\}.$ In particular, by Proposition~\ref{prop:containement}, we have $\mathbf{F}_\mathcal{C}(S)=\max_{\leq_{\mathbb{N}^d}}\mathcal{H}(S)$. Let $\preceq$ be any term order. If $\mathbf{h}\in \mathcal{H}(S)$, then $\mathbf{h}\leq_{\mathbb{N}^d} \mathbf{F}_\mathcal{C}(S)$. So, the inequality $\mathbf{h}\preceq \mathbf{F}_\mathcal{C}(S)$ holds. In particular, $\mathbf{F}_\mathcal{C}(S)=\max_\preceq \mathcal{H}(S)$.
\end{proof} 

\begin{proposition}\label{prop:almostSymmetric} \label{prop:C-almost-symmetric}
 Let $S$ be a $\mathcal{C}$-semigroup. Then $S$ is almost symmetric if and only if $\mathbf{F}_{\mathcal{C}}(S) - \mathbf{f} \in \mathrm{PF}(S)$ for all $\mathbf{f} \in \mathrm{PF}(S)\setminus \{ \mathbf{F}_{\mathcal{C}}(S)\}.$ 
\end{proposition}
\begin{proof}
Suppose $S$ is almost symmetric. Then there exists a term order $\preceq$ such that $\mathbf{F}_\preceq(S) -\mathbf{f} \in \operatorname{PF}(S)$ for all $\mathbf{f}\in \operatorname{PF(S)}\setminus \{\mathbf{F}_\preceq (S)\}$. In particular, if $\mathbf{h}\in \mathcal{H}(S)\cap \operatorname{PF}(S)$, then $\mathbf{h}\leq_{\mathcal{C}} \mathbf{F}_\preceq (S)$. Moreover, by \cite[Proposition 4.7]{op2}, if $\mathbf{h}\in \mathcal{H}(S)\setminus \operatorname{PF}(S)$ we have $\mathbf{F}_\preceq(S)-\mathbf{h}\in S$, that is, $\mathbf{h}\leq_\mathcal{C} \mathbf{F}_\preceq(S)$. Therefore, $\mathbf{F}_\preceq(S)=\mathbf{F}_\mathcal{C}(S)$ and hence $\mathbf{F}_{\mathcal{C}}(S) - \mathbf{f} \in \mathrm{PF}(S)$ for all $\mathbf{f} \in \mathrm{PF}(S)\setminus \{ \mathbf{F}_{\mathcal{C}}(S)\}.$ Conversely, by Proposition~\ref{prop:containement}, we have $\mathbf{F}_\mathcal{C}(S)=\max_{\leq_{\mathbb{N}^d}}\mathcal{H}(S)$. Let $\preceq$ be any term order. If $\mathbf{h}\in \mathcal{H}(S)$, then $\mathbf{h}\leq_{\mathbb{N}^d} \mathbf{F}_\mathcal{C}(S)$. So, the inequality $\mathbf{h}\preceq \mathbf{F}_\mathcal{C}(S)$ holds. In particular, $\mathbf{F}_\mathcal{C}(S)=\max_\preceq \mathcal{H}(S)$.
\end{proof}

\begin{remark} \rm
  By Propositions~\ref{prop:C-symmeric} and  \ref{prop:C-almost-symmetric}, we can argue that if $S$ is symmetric or almost symmetric $\mathcal{C}$-semigroup, then $\mathbf{F}_\mathcal{C}(S)=\mathbf{F}_\preceq(S)$ for all term orders $\preceq$. In particular, these notions depend on the semigroup only, not on the choice of a term order.  
\end{remark}

Recall that a $\mathcal{C}$-semigroup is said to be \emph{irreducible} if it is not possible to express it as an intersection of two $\cC$-semigroups properly containing it. In particular, every $\cC$-semigroup can be expressed as an intersection of finitely many irreducible $\cC$-semigroups. Moreover, it is known that if $S$ is an irreducible $\cC$-semigroup, then $S$ is symmetric or almost symmetric with $\operatorname{t}(S)=2$ (in this case, $S$ is also called \emph{pseudo-symmetric}). In particular, when $S$ is irreducible then $|\operatorname{SG}(S)|=1$ (see \cite{garcia-pseudofrobenius}). As a consequence, $\mathrm{Maximals}_{\leq_\cC}\cH(S)=\operatorname{SG}(S)$. In the next result, we consider a more general property.

\begin{theorem} \label{thm:equivalent-quasiIrred}
    Let $S$ be a $\mathcal{C}$-semigroup. Then the following are equivalent:

    \begin{enumerate}
        \item[(i)] $ \mathrm{Maximals}_{\leq_\cC} \cH(S)=\operatorname{SG}(S).$
        
      \item[(ii)] For all $\mathbf{f}\in \operatorname{PF}(S)$, then $\mathbf{f}\in \mathrm{Maximals}_{\leq_\cC} \cH(S)$ or $2\mathbf{f}\in \mathrm{Maximals}_{\leq_\cC} \cH(S)$. 

      \item[(iii)] For all $\mathbf{h}\in \mathcal{H}(S)$, $2\mathbf{h} \in \mathrm{Maximals}_{\leq_\cC} \cH(S)$ or there exists $\mathbf{F}\in \mathrm{Maximals}_{\leq_\cC} \cH(S)$ such that $\mathbf{F}-\mathbf{h}\in S$.

    \end{enumerate}
\end{theorem}
\begin{proof}
    $(1) \Rightarrow (2)$. Assume $ \mathrm{Maximals}_{\leq_\cC} \cH(S)=\operatorname{SG}(S)$ and let $\mathbf{f}\in \operatorname{PF}(S)$. If $2\mathbf{f}\in S$, then $\mathbf{f}\in \operatorname{SG}(S)$ and thus, $\mathbf{f}\in \mathrm{Maximals}_{\leq_\cC} \cH(S)$. So, assume $2\mathbf{f}\notin S$. In this case, observe that $2\mathbf{f}\in \operatorname{PF}(S)$. In fact, if $\mathbf{s}\in S\setminus \{\mathbf{0}\}$, then $2\mathbf{f}+\mathbf{s}=\mathbf{f}+(\mathbf{f}+\mathbf{s})$, and the claim holds by $\mathbf{f}\in \operatorname{PF}(S)$. Now, if $4\mathbf{f}\in S$, then $2\mathbf{f}\in \operatorname{SG}(S)=\mathrm{Maximals}_{\leq_\cC} \cH(S)$. Assume $4\mathbf{f}\notin S$. Since $S$ is a $\cC$-semigroup, we can define $k=\max\{i\in \mathbb{N}\mid i\mathbf{f}\notin S\}$ and by our assumption, $k\geq 3$. Note that if $j\mathbf{f} \notin S$ then $(j-1)\mathbf{f} \notin S$ for all $2\leq j\leq k$. Now, consider the gap $(k-1)\mathbf{f}$ and observe that it is a pseudo-Frobenius element. Moreover, $2(k-1)\mathbf{f}\in S$, since $2(k-1)\geq k+1$. Hence, $2(k-1)\mathbf{f}\in \operatorname{SG}(S)= \mathrm{Maximals}_{\leq_\cC} \cH(S)$. But also $k\mathbf{f}\in \mathcal{H}(S)$ and $(k-1)\mathbf{f}\leq_\cC k\mathbf{f}$, a contradiction. Therefore, $4\mathbf{f}\in S$.

    \noindent $(2) \Rightarrow (3)$. Let $\mathbf{h}\in \cH(S)$. If $\mathbf{h}\in \operatorname{PF}(S)$, then $2\mathbf{h}\in \mathrm{Maximals}_{\leq_\cC} \cH(S)$ or $\mathbf{h}\in \mathrm{Maximals}_{\leq_\cC} \cH(S)$. In the first case, we trivially conclude. The second case is also trivial. since we can consider that $\mathbf{h}-\mathbf{h}=\mathbf{0}\in S$. Suppose $\mathbf{h}\notin \operatorname{PF}(S)$. Then there exists $\mathbf{f}\in \operatorname{PF}(S)$ such that $\mathbf{h}\leq_S \mathbf{f}$. In particular $\mathbf{f}-\mathbf{h}\in S$. So, by hypothesis, $\mathbf{f}\in \mathrm{Maximals}_{\leq_\cC} \cH(S)$ or $2\mathbf{f}\in \mathrm{Maximals}_{\leq_\cC} \cH(S)$. In the first case, we are done, while in the second, we can observe that $2\mathbf{f}\in \operatorname{PF}(S)$ and $\mathbf{h}\leq_S 2\mathbf{f}$, that is, $2\mathbf{f}-\mathbf{h}\in S$.

    \noindent $(3)\Rightarrow(1)$ By Proposition~\ref{prop:containement} we know that $\mathrm{Maximals}_{\leq_\cC} \cH(S)\subseteq \operatorname{SG}(S)$. So, let $\mathbf{h}\in \operatorname{SG}(S)$ and assume $\mathbf{h}\notin \mathrm{Maximals}_{\leq_\cC} \cH(S)$. Since $2\mathbf{h}\in S$, by hypothesis there exists $\mathbf{F}\in  \mathrm{Maximals}_{\leq_\cC} \cH(S)$, with $\mathbf{F}\neq \mathbf{h}$, such that $\mathbf{F}-\mathbf{h}\in S$. But this means that there exists $\mathbf{s}\in S\setminus \{\mathbf{0}\}$ such that $\mathbf{h}+\mathbf{s}\in \cH(S)$. contradicting $\mathbf{h}\in \operatorname{SG}(S)$. Therefore, $\mathbf{h}\in \mathrm{Maximals}_{\leq_\cC} \cH(S)$.
\end{proof}

\begin{proposition}\label{prop:quasi-symmetric}
    Let $S$ be a $\mathcal{C}$-semigroup. Then $\mathrm{Maximals}_{\leq_\cC} \cH(S)=\operatorname{PF}(S)$ if and only if for all $\mathbf{h}\in \mathcal{H}(S)$, there exists $\mathbf{F}\in \mathrm{Maximals}_{\leq_\cC} \cH(S)$ such that $\mathbf{F}-\mathbf{h}\in S$.
\end{proposition}
\begin{proof}
  $(\Rightarrow)$
  If $\mathbf{h}\in \cH(S)$, then there exists $\mathbf{f}\in \operatorname{PF}(S)$ such that $\mathbf{h}\leq_S \mathbf{f}$, that is, $\mathbf{f}-\mathbf{h}\in S$ . Since, by hypothesis, $\mathbf{f}\in \mathrm{Maximals}_{\leq_\cC} \cH(S) $, and we are done.

  \noindent $(\Leftarrow)$ By Proposition~\ref{prop:containement} we know that $\mathrm{Maximals}_{\leq_\cC} \cH(S)\subseteq \operatorname{PF}(S)$. So, let $\mathbf{h}\in \operatorname{PF}(S)$. By hypothesis, there exists $\mathbf{F}\in \mathrm{Maximals}_{\leq_\cC} \cH(S) $ such that $\mathbf{F}-\mathbf{h}\in S$. This means $\mathbf{h}\leq_S \mathbf{F}$ and since $\mathbf{h}\in \operatorname{PF}(S)$ we have $\mathbf{F}=\mathbf{h}$.
\end{proof}

Taking into account the previous results, we say that a $\cC$-semigroup is \emph{quasi-irreducible} if it satisfies one of the equivalent conditions of Theorem~\ref{thm:equivalent-quasiIrred}, while we say it is \emph{quasi-symmetric} if it satisfies one of the equivalent conditions of Proposition~\ref{prop:quasi-symmetric}. Observe that these definitions are extensions of the same notions provided in \cite{singhal} for generalized numerical semigroups.

\begin{corollary}
    Let $S$ be a quasi-irreducible $\cC$-semigroup and denote $\tau(S)=|\mathrm{Maximals}_{\leq_\cC}\cH(S)|$.   Then $\tau(S) \leq \operatorname{t}(S) \leq 2\tau(S)$.
\end{corollary}

\section{Reduced type for $\mathcal{C}$-semigroups}

If $S$ is a numerical semigroup, we know that $K[S]$ is a Cohen-Macaulay one dimensional domain, with $\operatorname{t}(S)=|\mathrm{Maximals}_{\leq_S}\mathrm{Ap}(S,\operatorname{m}(S))|=|\operatorname{PF}(S)|$ and $\mathfrak{C}_S=\{n\in \mathbb{N}\mid n>\operatorname{F}(S)\}$. In this case, one can study the so-called \emph{reduced type}, introduced by \cite{huneke2021torsion} to study Berger's conjecture. This invariant is also studied in \cite{maitra2023extremal}. The reduced type is defined as $\mathrm{s}(R)=\dim_K\left(\frac{\mathfrak{C}_R+\mathfrak{q} R}{\mathfrak{q} R}\right)$ where $\mathfrak{q}$ is a minimal reduction of $R$. In particular, in this case, one can consider $\mathfrak{q}=(X^{\operatorname{m}(S)})$. In terms of the numerical semigroup $S$, it can be translated as $\mathrm{s}(R)=|\mathfrak{C}_S \cap \mathrm{Ap}(S,\operatorname{m}(S))|=|\cH(S) \cap [\operatorname{F}(S)-\operatorname{m}(S)+1,\operatorname{F}(S)]|$. In this case, we also have $\mathfrak{C}_S \cap \mathrm{Ap}(S,\operatorname{m}(S))\subseteq \mathrm{Maximals}_{\leq_S}\mathrm{Ap}(S,\operatorname{m}(S))$. So, it is not difficult to see that $1\leq \mathrm{s}(R)\leq \operatorname{t}(S)$.
From the point of view of invariants of an affine semigroup, the reduced type gives us the idea to investigate the relation between the conductor of an affine semigroup $S\subseteq \mathbb{N}^d$ and the Ap\'ery set $\mathrm{Ap}(S, Q)$, where $\operatorname{t}(S)=|\mathrm{Maximals}_{\leq_S}\mathrm{Ap}(S, Q)|$. The following example highlights a particular difference between numerical semigroups and affine semigroups in $\mathbb{N}^d$, with $d>1$, for what concerns the relationship of the Ap\'ery set and conductor.

\medskip
\begin{example} \rm
   \noindent  Let $S=\langle (0,4),(2,0),(1,1),(1,2)\rangle\subseteq \mathbb{N}^2$. $S$ is a simplicial Cohen-Macaulay affine semigroup with $E=\{(0,4),(2,0)\}$ and $\overline{S}=\mathbb{N}^2$. It is possible to compute $\operatorname{Ap}(S,E)=\{(0,0),(1,2),(1,1),(2,2),(3,3),(2,3),(3,4),(4,5)\}$. Moreover, $(3,2)+\mathbb{N}^2\subseteq \mathfrak{C}_S$. Hence, we have $(3,4)\in \mathfrak{C}_S \cap \mathrm{Ap}(S,E)$ but $(3,4)\notin \mathrm{Maximals}_{\leq_S}\mathrm{Ap}(S,E)$, since $(3,4)\leq_S (4,5)$.
\end{example}


So, it can occur that $\mathfrak{C}_S \cap \mathrm{Ap}(S,Q)\nsubseteq \mathrm{Maximals}_{\leq_S}\mathrm{Ap}(S,Q)$. This means that the cardinality of $\mathfrak{C}_S \cap \mathrm{Ap}(S,Q)$ can not be straightforwardly interpreted as a lower bound of $\operatorname{t}(S)$. For this reason, we suggest to investigate the properties of the following value:
$$\operatorname{s}(S,Q)=|\mathfrak{C}_S \cap \mathrm{Maximals}_{\leq_S}\mathrm{Ap}(S,Q)|.$$

In this way, we have trivially $\operatorname{s}(S,Q)\leq \operatorname{t}(S)$. Observe also that, if $S$ is a numerical semigroup, then $Q=\{\operatorname{m}(S)\}$ and $\mathfrak{C}_S \cap \mathrm{Maximals}_{\leq_S}\mathrm{Ap}(S,\operatorname{m}(S))=\mathfrak{C}_S \cap \mathrm{Ap}(S,\operatorname{m}(S))$. So, we have that $\operatorname{s}(S,\operatorname{m}(S))$ is the reduced type of $S$ as defined in \cite{maitra2023extremal}. Another difference with respect to numerical semigroups, in the case $S$ is an affine semigroup in $\mathbb{N}^d$ with $d>1$, is the possibility that $\operatorname{s}(S,Q)=0$, as we show in the following example:

\begin{example} \rm
    Let $S\subseteq \mathbb{N}^2$ be the affine semigroup generated by the set:
    $$ A=\{(3,0),(0,4), (1,6),(2,6),(2,8),(2,9),(3,6),(3,7),(4,3),(4,4),(4,5),(5,3),(6,5)\}.$$
\noindent    $S$ is a Cohen-Macaulay simplicial affine semigroup with $E=\{(0,4),(3,0)\}$ and $\overline{S}=\mathbb{N}^2$. Moreover $\mathfrak{C}_S=(4,6)+\mathbb{N}^2$. In fact one can see that $(4,6)+\mathbb{N}^2\subseteq S$ since $[(4,6),(7,10)]\subseteq S$. Furthermore, $\{(5+3n,5)\in \mathbb{N}^2\mid n\geq 0\}\cup\{(3,9+4n)\in \mathbb{N}^2 \mid n\geq 0\}\subseteq \cH(S)$. From this, one can argue that if $\mathbf{s}\notin (4,6)+\mathbb{N}^2$, then there exits $\mathbf{n} \in \mathbb{N}^2$ such that $\mathbf{s}+\mathbf{n}\notin S$. It is possible also to compute that $\operatorname{Ap}(S,E)=A\cup \{(0,0)\}$. 
 Therefore, $\mathfrak{C}_S \cap \mathrm{Maximals}_{\leq_S}\mathrm{Ap}(S,E)=\emptyset$.
\end{example}

Now, we investigate some properties of the value $\operatorname{s}(S,Q)$ defined above, in the case of $\mathcal{C}$-semigroups. In particular, we study its extremal behavior, extending some results contained in \cite{maitra2023extremal}. Let $S$ be a $\mathcal{C}$-semigroup. We can write $S=\langle \mathbf{n}_1,\ldots,\mathbf{n}_e, \mathbf{n}_{e+1},\ldots,\mathbf{n}_{e+r}\rangle \subseteq \mathbb{N}^d$ where $E=\{\mathbf{n}_1,\ldots,\mathbf{n}_e\}$ is the set of minimal extremal rays of $\mathcal{C}$.
It is known that in this case $\mathrm{depth}(K[S])=1$ (see \cite{garcia-pseudofrobenius}). Therefore, $X^{\mathbf{n_i}}$ is a maximal regular sequence in $K[S]$. So, for each $\mathbf{n}_i \in E$, we consider
$$\operatorname{s}(S,\mathbf{n}_i)=|\mathfrak{C}_S \cap \mathrm{Maximals}_{\leq_S}\mathrm{Ap}(S,\mathbf{n}_i)|,$$ 
and we call it the \emph{reduced type} of $S$ with respect to $\mathbf{n}_i$. We start to provide some useful properties to study the reduced type in this special context.

\begin{proposition}\label{prop:conductor}
    Let $S$ be a $\mathcal{C}$-semigroup. Then $$\mathfrak{C}_S = \{\mathbf{x} \in S \mid \mathbf{x} \nleq_{\mathcal{C}} \mathbf{F} \text{ for all } \mathbf{F} \in \mathrm{Maximals}_{\leq_{\mathcal{C}}}\mathcal{H}(S)\}.$$
\end{proposition}
\begin{proof}
    Let $\mathbf{x} \in S$ such that $\mathbf{x}\nleq_{\mathcal{C}} \mathbf{F}$ for all $\mathbf{F} \in \mathrm{Maximals}_{\leq_{\mathcal{C}}}\mathcal{H}(S)$. We need to prove $\mathbf{x} + \mathbf{c} \in S$ for all $\mathbf{c} \in \mathcal{C}$. Suppose $\mathbf{x} + \mathbf{c} \notin S$ for some $\mathbf{c} \in \mathcal{C}$. This implies $\mathbf{x} + \mathbf{c} \in \mathcal{H}(S)$. Therefore, there exists some $\mathbf{F} \in \mathrm{Maximals}_{\leq_{\mathcal{C}}}\mathcal{H}(S)$ such that $\mathbf{x} + \mathbf{c}\leq_{\mathcal{C}} \mathbf{F}$. This implies $\mathbf{x} + \mathbf{c}+\mathbf{d} = \mathbf{F}$ for some $\mathbf{d} \in \mathcal{C}$. Since $\mathbf{c}+\mathbf{d} \in \mathcal{C}$, we get $\mathbf{x} \leq_{\mathcal{C}} \mathbf{F}$, a contradiction. Hence, $\mathbf{x} \in \mathfrak{C}_S.$ Conversely, suppose $\mathbf{x} \in \mathfrak{C}_S$, then $\mathbf{x} \in S$, thus we need to prove $\mathbf{x}\nleq_{\mathcal{C}} \mathbf{F}$ for all $\mathbf{F} \in \mathrm{Maximals}_{\leq_{\mathcal{C}}}\mathcal{H}(S)$. If there exists some $\mathbf{F} \in \mathrm{Maximals}_{\leq_{\mathcal{C}}}\mathcal{H}(S)$ such that $\mathbf{x}\leq_{\mathcal{C}} \mathbf{F}$ then $\mathbf{F} = \mathbf{x} + \mathbf{c} $ for some $\mathbf{c} \in \mathcal{C}$. This implies $\mathbf{x} + \mathbf{c} \notin S$. This is a contradiction.
\end{proof}

\medskip

\begin{lemma}\label{s(S,n_i)non_empty}
Let $S$ be a $\mathcal{C}$-semigroup. If $\mathbf{h}$ is a maximal gap with respect to the order $\leq_\mathcal{C}$, then $\mathbf{h}+\mathbf{n}_i\in \mathfrak{C}_S \cap \mathrm{Maximals}_{\leq_S} \mathrm{Ap}(S,\mathbf{n}_i)$ for all $i\in\{1,\ldots,e\}$. In particular, $\operatorname{s}(S,\mathbf{n}_i)\geq 1$ for all $i\in \{1,\ldots,e\}$.
\end{lemma}
\begin{proof}
 Since $\mathrm{Maximals}_{\leq_{\mathcal{C}}}\mathcal{H}(S) \subseteq \mathrm{Maximals}_{\leq_S}\mathcal{H}(S)$, we get $\mathbf{h} + \mathbf{n}_i \in  \mathrm{Maximals}_{\leq_S} \mathrm{Ap}(S,\mathbf{n}_i).$  Also, since $\mathbf{h}+\mathbf{n}_i+\mathbf{c} \in \mathcal{C}$ for all $c \in \mathcal{C}$ and $h \in \mathrm{max}_{\leq_{\mathcal{C}}}\mathcal{H}(S)$, we get $\mathbf{h} + \mathbf{n}_i + \mathbf{c} \in S$ for all $\mathbf{c} \in \mathcal{C}$. Thus, we get $\mathbf{h}+\mathbf{n}_i\in \mathfrak{C}_S \cap \mathrm{Maximals}_{\leq_S} \mathrm{Ap}(S,\mathbf{n}_i)$.
\end{proof}
\medskip
We say that $S$ has \emph{minimal reduced type} if $\operatorname{s}(S,\mathbf{n}_i)=1$ for all $i\in \{1,\ldots, e\}$. We say that $S$ has \emph{maximal reduced type} if $s(S,\mathbf{n}_i)=\operatorname{t}(S)$ for all $i\in \{1,\ldots,e\}$.

\begin{proposition}\label{prop:maximal-redtype}
    Let $S$ be a $\mathcal{C}$-semigroup. Then $S$ has maximal reduced type if and only if for all $\mathbf{f}\in \operatorname{PF}(S)$ and for all $i\in \{1,\ldots,e\}$ we have $\mathbf{f}+\mathbf{n}_i\nleq_{\mathcal{C}} \mathbf{F}$ for all $\mathbf{F} \in \mathrm{Maximals}_{\leq_\cC} \cH(S)$.
\end{proposition}
\begin{proof}
    By Proposition~\ref{prop:conductor}, we know that $\mathbf{f}+\mathbf{n}_i\in \mathfrak{C}_S$ if and only if $\mathbf{f}+\mathbf{n}_i\nleq_{\mathcal{C}} \mathbf{F}$ for all $\mathbf{F} \in \mathrm{Maximals}_{\leq_\cC} \cH(S)$. On the other hand, we know by Proposition~\ref{prop:aperyPF} that $\mathbf{f}+\mathbf{n}_i \in \mathrm{Maximals}_{\leq_S} \mathrm{Ap}(S,\mathbf{n}_i)$ for all $i\in\{1,\ldots,e\}.$ This completes the proof. 
\end{proof}

By Lemma~\ref{s(S,n_i)non_empty}, if $S$ has a minimal reduced type, then there exists only one maximal element in the set of gaps with respect to natural partial order $\leq_{\mathcal{C}}$. 

\medskip
By Proposition~\ref{prop:C-symmeric}, it is easy to see that all symmetric $\mathcal{C}$-semigroups are of minimal reduced type. As we will show in Example~\ref{exa: minimalRed-notSymmetric}, the converse is not true.

\begin{proposition}\label{equivalent_minreltype}
    Let $S$ be a $\mathcal{C}$-semigroup. Then $S$ has minimal reduced type if and only if there exists $\mathbf{F}_\mathcal{C}(S)$ and for all $\mathbf{f}\in \operatorname{PF}(S)\setminus \{\mathbf{F}_\mathcal{C}(S)\}$ and for all $i\in \{1,\ldots,e\}$ we have $\mathbf{f}\leq_{\mathcal{C}} \mathbf{F}_\mathcal{C}(S)-\mathbf{n}_i$.
\end{proposition}
\begin{proof}
    Suppose $S$ has minimal reduced type and there exist $\mathbf{h}_1 \neq \mathbf{h}_2 \in \mathrm{max}_{\leq_{\mathcal{C}}}\mathcal{H}(S)$. By Lemma \ref{s(S,n_i)non_empty}, $\{\mathbf{h}_1+\mathbf{n}_i,\mathbf{h}_2+\mathbf{n}_i\} \subseteq \mathfrak{C}_S \cap \mathrm{Maximals}_{\leq_S} \mathrm{Ap}(S,\mathbf{n}_i)$ for all $i \in \{1,\ldots,e\}.$ This is the contradiction to the minimal reduced type. Hence, we get $\{\mathbf{F}_\mathcal{C}(S)\}=\mathrm{max}_{\leq_{\mathcal{C}}}\mathcal{H}(S)$. Now, consider $\mathbf{f} \in \operatorname{PF}(S)\setminus \{\mathbf{F}_\mathcal{C}(S)\}$. Observe that $\mathbf{F}_\mathcal{C}(S)+\mathbf{n}_i\in \mathfrak{C}_S \cap \mathrm{Maximals}_{\leq_S} \mathrm{Ap}(S,\mathbf{n}_i)$, for all $i\in \{1,\ldots,e\}$. So, since $S$ has minimal reduced type, we get $\mathbf{f} + \mathbf{n}_i \notin \mathfrak{C}_S$ for all $i \in \{1, \ldots, e\}$ because $\mathbf{f}+\mathbf{n}_i \in \mathrm{Maximals}_{\leq_S} \mathrm{Ap}(S,\mathbf{n}_i)$ for all $i \in \{1, \ldots, e\}.$ Therefore, there exists some $\mathbf{c} \in \mathcal{C}$ such that $\mathbf{f} + \mathbf{n}_i + \mathbf{c} \notin S.$ Since $\mathbf{f} + \mathbf{n}_i + \mathbf{c} \in \mathcal{C}$, we get $\mathbf{f} + \mathbf{n}_i + \mathbf{c} \in \mathcal{H}(S).$ Therefore, we get $\mathbf{f} + \mathbf{n}_i + \mathbf{c} \leq_{\mathcal{C}} \mathbf{F}_\mathcal{C}(S)$. Hence, $\mathbf{f}\leq_{\mathcal{C}} \mathbf{F}_\mathcal{C}(S)-\mathbf{n}_i$ for all $i \in \{1, \ldots, e\}.$ To prove the converse, it is sufficient to prove that $\mathfrak{C}_S \cap \mathrm{Maximals}_{\leq_S}\mathrm{Ap}(S,\mathbf{n}_i) = \{\mathbf{F}_\mathcal{C}(S) + \mathbf{n}_i\} $ for all $i \in \{1, \ldots, e\}.$ This is true because for all $\mathbf{f}\in \operatorname{PF}(S)\setminus \{\mathbf{F}_\mathcal{C}(S)\}$ and for all $i\in \{1,\ldots,e\}$, we have $\mathbf{f}\leq_{\mathcal{C}} \mathbf{F}_\mathcal{C}(S)-\mathbf{n}_i$, and thus $\mathbf{f} + \mathbf{n}_i \notin \mathfrak{C}_S$ for all $\mathbf{f}\in \operatorname{PF}(S)\setminus \{\mathbf{F}_\mathcal{C}(S)\}$ and for all $i\in \{1,\ldots,e\}.$
\end{proof}

Now, we provide some necessary conditions for maximal and minimal reduced type in the case of almost symmetric or quasi-irreducible $\mathcal{C}$-semigroups.

\begin{proposition}\label{prop:minimalRed_almostSym}
    Let $S$ be an almost symmetric $\mathcal{C}$-semigroups. The following holds:
    \begin{enumerate}
        \item[(i)] If $S$ is minimal reduced type, then $ \mathbf{n}_i+\mathbf{n}_j\leq_{\mathcal{C}} \mathbf{F}_\mathcal{C}(S)$ for all $i,j\in \{1,\ldots,e\}$.
        \item[(ii)] If $S$ is maximal reduced type, then for all $i \in \{1,\ldots,e\}$ and for all $\mathbf{f}\in \operatorname{PF}(S)\setminus \{\mathbf{F}_\cC(S)\}$ we have $\mathbf{n}_i\nleq_\cC \mathbf{f}$.
    \end{enumerate}
\end{proposition}
\begin{proof}
    By Proposition \ref{equivalent_minreltype}, if $\mathbf{f}\in \operatorname{PF}(S)$ we have $\mathbf{F}_\mathcal{C}(S) - \mathbf{f} \leq_{\mathcal{C}} \mathbf{F}_\mathcal{C}(S) - \mathbf{n}_i$ for all $i \in \{1, \ldots,e\}$. This implies $\mathbf{n}_i \leq_{\mathcal{C}} \mathbf{f} $ for all $\mathbf{f}\in \operatorname{PF}(S)\setminus \{\mathbf{F}_\mathcal{C}(S)\}.$ Thus, we get $\mathbf{n}_j \leq_{\mathcal{C}} \mathbf{f} \leq_{\mathcal{C}} \mathbf{F}_\mathcal{C}(S)-\mathbf{n}_i$ for all $i,j\in \{1,\ldots,n\}$. This completes the proof of (i). To prove (ii), let $\mathbf{f}\in \operatorname{PF}(S)\setminus \{\mathbf{F}_\cC(S)\}$ and $i\in \{1,\ldots,e\}$. Since $S$ is almost symmetric, we have we have $\mathbf{F}_\cC(S)-\mathbf{f}\in \operatorname{PF}(S)$. Moreover, since $S$ has maximal reduced type, we have $(\mathbf{F}_\cC(S)-\mathbf{f})+\mathbf{n}_i\nleq_\cC \mathbf{F}_\cC(S)$. If we assume $\mathbf{n}_i\leq_\cC \mathbf{f}$, then we have $\mathbf{F}_\cC(S)+\mathbf{n}_i \leq _\cC \mathbf{F}_\cC(S)+\mathbf{f}$. In particular, we obtain $(\mathbf{F}_\cC(S)-\mathbf{f})+\mathbf{n}_i\leq_\cC \mathbf{F}_\cC(S)$, a contradiction. So, $\mathbf{n}_i \nleq_\cC \mathbf{f}$. This completes the proof.  
    \end{proof}

    \begin{example}\label{exa: minimalRed-notSymmetric} \rm
     Let $S\subseteq \mathbb{N}^2$ be the generalized numerical semigroup having the following set of gaps:     \begin{align*}
    \cH(S)= 
    \left\lbrace
    \begin{array}{c}
     ( 0, 1 ), ( 0, 2 ), ( 0, 4 ), ( 0, 5 ), ( 0, 7 ), ( 0, 8 ), 
    ( 1, 0 ), ( 1, 1 ), ( 1, 2 ), ( 1, 3 ), ( 1, 4 ), ( 1, 5 ), \\ [1mm] ( 1, 6 ), 
  ( 1, 7 ),( 1, 8 ), ( 2, 0 ), ( 2, 1 ), ( 2, 2 ), ( 2, 3 ), ( 2, 4 ), ( 2, 5 ), ( 2, 6 ), 
  ( 2, 7 ), ( 2, 8 ), \\ [1mm] ( 3, 1 ), ( 3, 2 ),  
  ( 3, 4 ), ( 3, 5 ), ( 3, 7 ), ( 3, 8 ),  ( 4, 0 ), ( 4, 1 ), ( 4, 2 ), ( 4, 3 ), \\[1mm] ( 4, 4 ), ( 5, 2 ), ( 5, 5 ), ( 5, 8 ), ( 8, 2 ), 
  ( 8, 5 ), ( 8, 8 ) 
   \end{array}
  \right\rbrace.
        \end{align*}
        With the help of the \texttt{GAP}\cite{GAP} package \texttt{numericalsgp}\cite{NumericalSgps}, one can check that $S$ is really a generalized numerical semigroup and that $\operatorname{PF}(S)=\{(8,8), (4,4)\}$. Moreover, in this example, the minimal extreme rays are $\mathbf{n}_1=(0,3)$ and $\mathbf{n}_2=(3,0)$ and $\mathbf{F}_{\mathbb{N}^d}(S)=(8,8)$. So, it is easy to check that $S$ has minimal reduced type by Proposition~\ref{equivalent_minreltype}. Moreover, $S$ is almost symmetric but not symmetric, and one can check that Proposition~\ref{prop:minimalRed_almostSym}(i) holds. 
        
        \noindent Furthermore, let $S=\mathbb{N}^2\setminus \{(0,1),(1,0),(1,1),(2,1)\}$. In this case, the set of minimal extremal rays is $\{(0,2),(2,0)\}$ and we have $\operatorname{PF}(S)=\{(1,0),(1,1),(2,1)\}$. In particular, $S$ has maximal reduced type and is also almost symmetric. In particular, one can check that Proposition~\ref{prop:minimalRed_almostSym}(ii) holds. This example shows also that the condition $\mathbf{n}_i\nleq_\cC \mathbf{f}$ cannot be extended to $\mathbf{f}=\mathbf{F}_\cC(S)$. In fact, in this case $\cC=\mathbb{N}^2$, $\mathbf{F}_{\mathbb{N}^2}(S)=(2,1)$ and $(2,0)\leq_{\mathbb{N}^2} \mathbf{F}_{\mathbb{N}^2}(S)$.
    \end{example}

\begin{proposition}\label{prop:quasiIrred-maxRed}
    Let $S$ be a quasi-irreducible $\cC$-semigroup. Then $S$ is maximal reduced type if and only if for all $\mathbf{f}\in \mathrm{Maximals}_{\leq_\cC}\cH(S)$ such that $\frac{\mathbf{f}}{2}\in \operatorname{PF}(S)$ we have $\frac{\mathbf{f}}{2}+\mathbf{n}_i\nleq_\cC \mathbf{F}$ for all $\mathbf{F}\in \mathrm{Maximals}_{\leq_\cC}\cH(S)$ and for all $i\in \{1,\ldots,e\}$.
\end{proposition}
\begin{proof}
   Let us define the set $$T=\left \lbrace\mathbf{f}\in \mathrm{Maximals}_{\leq_\cC}\cH(S)\mid \frac{\mathbf{f}}{2}\in \operatorname{PF}(S)\right \rbrace.$$
   
   \noindent $(\Rightarrow)$. Suppose $S$ has maximal reduced type. Then, by Proposition~\ref{prop:maximal-redtype}, for all $\mathbf{h}\in \operatorname{PF}(S)$ we have $\mathbf{h}+\mathbf{n}_i\nleq_\cC \mathbf{F}$ for all $\mathbf{F}\in \mathrm{Maximals}_{\leq_\cC}\cH(S)$ and for all $i\in \{1,\ldots,e\}$. So, this implication is trivial.

   \noindent $(\Leftarrow)$. Let $\mathbf{h}\in \operatorname{PF}(S)$. Using again Proposition~\ref{prop:maximal-redtype}, we want to show that $\mathbf{h}+\mathbf{n}_i\nleq_\cC \mathbf{F}$ for all $\mathbf{F}\in \mathrm{Maximals}_{\leq_\cC}\cH(S)$ and for all $i\in \{1,\ldots,e\}$. Now, if $\mathbf{h}\in T$ this property hold by hypothesis. So suppose $\mathbf{h}\notin T$. Since $S$ is quasi-irreducible, by (ii) of Theorem~\ref{thm:equivalent-quasiIrred}, we have that $\mathbf{h}\in \mathrm{Maximals}_{\leq_\cC}\cH(S)$. So, by Lemma~\ref{s(S,n_i)non_empty}, $\mathbf{h}+\mathbf{n}_i\in \mathfrak{C}_S$ for all $i\in \{1,\ldots,e\}$. Therefore, by Proposition~\ref{prop:conductor}, $\mathbf{h}+\mathbf{n}_i \nleq_\cC \mathbf{F}$ for all $\mathbf{F}\in \mathrm{Maximals}_{\leq_\cC}\cH(S)$ and for all $i\in \{1,\ldots,e\}$. This completes the proof.
\end{proof}

\begin{corollary}
    Let $S$ be a quasi-symmetric $\mathcal{C}$-semigroup. Then $S$ has maximal reduced type.
\end{corollary}
\begin{proof}
Note that being quasi-symmetric is equivalent to $T = \emptyset$ in the Proposition~\ref{prop:quasiIrred-maxRed}, and the result follows.   
\end{proof}

We previously examined the property of having maximal reduced type for the class of quasi-irreducible $\cC$-semigroups. The following results explore the same property for some classes of generalized numerical semigroups obtained by some constructions.

\medskip
\noindent Let $d > 1$ be an integer and $T$ a numerical semigroup. Define the set $$S_T=\{\mathbf{x}=(x_1,\ldots,x_d)\in \mathbb{N}^d\mid \lVert \mathbf{x}\rVert_1 = x_1+\ldots+x_d \in T\}.$$
One can observe that $S_T$ is a generalized numerical semigroup, and it is called $T$-graded generalized numerical semigroup (see \cite{cisto-navarra}). 

\begin{theorem}\label{thm:tgraded_redtype}
   Let $d > 1$ be an integer and $T$ a numerical semigroup. Then $S_T$ has maximal reduced type if and only if $T$ has maximal reduced type. 
\end{theorem}
\begin{proof}
 It is known by \cite[Proposition 4.5]{cisto-navarra} that $\operatorname{PF}(S_T)=\{\mathbf{x}\in \mathbb{N}^d \mid \lVert\mathbf{x}\rVert_1\in \operatorname{PF}(T)\}$. We show that $\mathfrak{C}_{S_T}=\{\mathbf{x}\in \mathbb{N}^d \mid \lVert\mathbf{x}\rVert_1 >\operatorname{F}(T)\}$. First, observe that if $\lVert\mathbf{x}\rVert_1 >\operatorname{F}(T)$, then $\mathbf{x}\in S_T$ and if $\mathbf{y}\in \mathbb{N}^d$, then $\lVert \mathbf{x}+\mathbf{y}\rVert_1=\lVert \mathbf{x}\rVert_1+\lVert \mathbf{y}\rVert_1 >\operatorname{F}(T)$, so $\mathbf{x}+\mathbf{y}\in S_T$. Hence $\mathfrak{C}_{S_T}\supseteq \{\mathbf{x}\in \mathbb{N}^d \mid \lVert\mathbf{x}\rVert_1 >\operatorname{F}(T)\}$. In order to prove the other inclusion, let $\mathbf{x}\in \mathfrak{C}_{S_T}$ and we show that $\lVert\mathbf{x}\rVert_1>\operatorname{F}(T)$. Assume that $\lVert \mathbf{x}\rVert_1\leq \operatorname{F}(T)$. Since $\mathbf{x}\in S_T$, we can assume $\lVert \mathbf{x}\rVert_1< \operatorname{F}(T)$. Hence, there exists $n\in \mathbb{N}\setminus\{0\}$ such that $\lVert \mathbf{x}\rVert_1+n = \operatorname{F}(T)$. Then, there exists $\mathbf{y}\in \mathbb{N}^d$ such that $\lVert \mathbf{y}\rVert_1 =n$. In particular, $\lVert \mathbf{x}+\mathbf{y}\rVert_1=\operatorname{F}(S)$. So, $\mathbf{x}+\mathbf{y}\notin S_T$, a contradiction to $\mathbf{x}\in \mathfrak{C}_{S_T}$. Therefore, $\mathfrak{C}_{S_T}=\{\mathbf{x}\in \mathbb{N}^d \mid \lVert\mathbf{x}\rVert_1 >\operatorname{F}(T)\}$. Now, suppose $T$ has maximal reduced type, and we show that $S_T$ has maximal reduced type. First note that the set of minimal extremal rays of $S_T$ is $\{\operatorname{m}(T)\mathbf{e}_i\mid i\in \{1,\ldots,d\}\}$.
 Let $\mathbf{f}\in \operatorname{PF}(T)$, then $\lVert \mathbf{f}\rVert_1\in \operatorname{PF}(T)$. Since $T$ is maximal reduced type we have $\lVert \mathbf{f}\rVert_1+\operatorname{m}(T) > \operatorname{F}(T)$. In particular, if we take $\operatorname{m}(T)\mathbf{e}_i$ for any $i\in\{1,\ldots,d\}$, then $\lVert \mathbf{f}+\operatorname{m}(T)\mathbf{e}_i\rVert_1>\operatorname{F}(T)$. Hence, $\mathbf{f}+\operatorname{m}(T)\mathbf{e}_i\in \mathfrak{C}_{S_T}$ for all $\mathbf{f} \in \operatorname{PF}(S)$ and for all $i \in \{1,\ldots,d\}$. By Proposition~\ref{prop:conductor} and Proposition~\ref{prop:maximal-redtype} we obtain that $S_T$ has maximal reduced type. Conversely, suppose that $S_T$ has maximal reduced type and we prove $T$ has maximal reduce type. Let $f\in \operatorname{PF}(T)$, it is sufficient to prove that $f+\operatorname{m}(T) > \operatorname{F}(T)$. We know that $f\mathbf{e}_i\in \operatorname{PF}(S_T)$. Since $S_T$ has maximal reduced type, we have $f\mathbf{e}_i+\operatorname{m}(T)\mathbf{e}_i \nleq_{\mathbb{N}^d} \mathbf{F}$ for all $\mathbf{F}\in \operatorname{FA}(S_T)$. In particular, $f\mathbf{e}_i+\operatorname{m}(T)\mathbf{e}_i \nleq_{\mathbb{N}^d} \operatorname{F}(T)\mathbf{e}_i$. But observe that these two vectors are comparable with respect to $\leq_{\mathbb{N}^d}$, so we get $ \operatorname{F}(T)\mathbf{e}_i<_{\mathbb{N}^d} f\mathbf{e}_i+\operatorname{m}(T)\mathbf{e}_i$. In particular, $f+\operatorname{m}(T)>\operatorname{F}(T)$. Hence, $T$ has maximal reduced type. This completes the proof.
\end{proof} 

\begin{remark} \rm
 Let $d > 1$ be an integer and $T$ a numerical semigroup.  Then $S_T$ is never minimal reduced type since $\operatorname{FA}(S_T)=\{\mathbf{x}\in \mathbb{N}^d \mid \lVert \mathbf{x}\rVert_1 =\operatorname{F}(T)\}$ and this set has cardinality greater than 1.  
\end{remark}

\begin{example}
\rm
Let $T=\langle 5,6,7 \rangle=\mathbb{N}\setminus \{1,2,3,4,8,9\}$. Therefore $\operatorname{PF}(T)=\{8,9\}$. The semigroup $S_T\subseteq \mathbb{N}^2$ is pictured in Figure~\ref{fig:S_T}. The set of gaps is the set of points marked in black and blue. In particular, the blue points represent the pseudo-Frobenius elements of $S_T$. The minimal generators are marked in red. All points not marked belong to $S_T$.    

\begin{figure}[h!]
\begin{center}
\begin{tikzpicture}[scale=0.7] 
\usetikzlibrary{patterns}
\draw [help lines] (0,0) grid (10,10);
\draw [<->] (0,11) node [left] {$y$} -- (0,0)
-- (11,0) node [below] {$x$};
\foreach \i in {1,...,10}
\draw (\i,1mm) -- (\i,-1mm) node [below] {$\i$} 
(1mm,\i) -- (-1mm,\i) node [left] {$\i$}; 
\node [below left] at (0,0) {$O$};

\draw [mark=*] plot (0,1);
\draw [mark=*] plot (1,0);
\draw [mark=*] plot (0,2); 
\draw [mark=*] plot (1,1);
\draw [mark=*] plot (2,0);
\draw [mark=*] plot (3,0);
\draw [mark=*] plot (2,1);
\draw [mark=*] plot (1,2);
\draw [mark=*] plot (0,3);
\draw [red, mark=*] plot (5,0);
\draw [red, mark=*] plot (4,1);
\draw [red, mark=*] plot (3,2);
\draw [red, mark=*] plot (2,3);
\draw [red, mark=*] plot (1,4);
\draw [red, mark=*] plot (0,5);
\draw [blue, mark=*] plot (9,0);
\draw [blue, mark=*] plot (8,1);
\draw [blue, mark=*] plot (7,2);
\draw [blue, mark=*] plot (6,3);
\draw [blue, mark=*] plot (5,4);
\draw [blue, mark=*] plot (4,5);
\draw [blue, mark=*] plot (3,6);
\draw [blue, mark=*] plot (2,7);
\draw [blue, mark=*] plot (1,8);
\draw [blue, mark=*] plot (0,9);

\draw [blue, mark=*] plot (8,0);
\draw [blue, mark=*] plot (7,1);
\draw [blue, mark=*] plot (6,2);
\draw [blue, mark=*] plot (5,3);
\draw [blue, mark=*] plot (4,4);
\draw [blue, mark=*] plot (3,5);
\draw [blue, mark=*] plot (2,6);
\draw [blue, mark=*] plot (1,7);
\draw [blue, mark=*] plot (0,8);

\draw [mark=*] plot (4,0);
\draw [mark=*] plot (3,1);
\draw [mark=*] plot (2,2);
\draw [mark=*] plot (1,3);
\draw [mark=*] plot (0,4);
\draw [red, mark=*] plot (0,6);
\draw [red, mark=*] plot (1,5);
\draw [red, mark=*] plot (2,4);
\draw [red, mark=*] plot (3,3);
\draw [red, mark=*] plot (4,2);
\draw [red, mark=*] plot (5,1);
\draw [red, mark=*] plot (6,0);
\draw [red, mark=*] plot (7,0);
\draw [red, mark=*] plot (6,1);
\draw [red, mark=*] plot (5,2);
\draw [red, mark=*] plot (4,3);
\draw [red, mark=*] plot (3,4);
\draw [red, mark=*] plot (2,5);
\draw [red, mark=*] plot (1,6);
\draw [red, mark=*] plot (0,7);

\end{tikzpicture}
\end{center}
\caption{Picture of gaps (black and blue), minimal generators (red) and pseudo-Frobenius elements (blue) of $S_T$, with $T=\langle 5,6,7 \rangle$}
\label{figure-stripe}
\label{fig:S_T}
\end{figure}
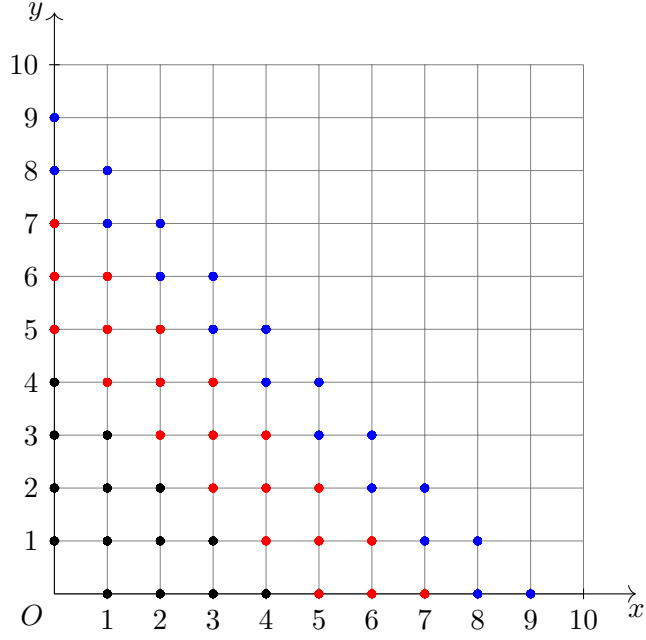    
\end{example}

Next, we consider the $k$-thickening of a generalized numerical semigroup introduced in \cite[Definition 3.6]{cisto2020generalization}. Let $\mathbf{e}_1,\ldots,\mathbf{e}_{d+1}$ be the standard generators of the semigroup $\mathbb{N}^{d+1}$ and consider the semigroup isomorphic to $\N^d$ inside $\N^{d+1}$ generated by $\{\mathbf{e}_1,\ldots,\mathbf{e}_{d+1}\}\setminus \mathbf{e}_i$ for some $i$.  By abuse of notation, we refer to the latter semigroup as $\N^d$.  Suppose $S\subseteq\N^d$ is a semigroup.  The $k$-thickening of $S\subseteq\N^d$ along axis $i$ in $\N^{d+1}$ is the semigroup $\kthick{k}(S,i)\subseteq\N^{d+1}$ defined as
\[
\kthick{k}(S,i)=S\cup (\mathbf{e}_i+S)\cup\cdots\cup (k\mathbf{e}_i+S) \cup ((k+1)\mathbf{e}_i+\N^{d+1}).
\]

\begin{proposition} \label{prop:kthick}
    Let $\kthick{k}(S,i)$ be the $k$-thickening of a generalized numerical semigroup $S\subseteq \mathbb{N}^d$. Then the following holds:
    \begin{itemize}
    \item[(i)] $\cH(\kthick{k}(S,i))= \cH(S)\cup (\mathbf{e}_i+\cH(S))\cup\cdots\cup (k\mathbf{e}_i+\cH(S))$, 
        \item[(ii)] $\operatorname{PF}(\kthick{k}(S,i))=k\mathbf{e}_i+\operatorname{PF}(S)$,
        \item[(iii)] $\mathfrak{C}_{\kthick{k}(S,i)}=((k+1)\mathbf{e}_i+\mathbb{N}^{d+1})\cup (\mathfrak{C}_S+\mathbb{N}^{d+1}).$
    \end{itemize}
\end{proposition}
\begin{proof}
    The proof of (i) follows immediately from the definition of $\kthick{k}(S,i)$. In order to prove (ii), let $\mathbf{f}\in \operatorname{PF}(\kthick{k}(S,i))$. Write $\mathbf{f}=j\mathbf{e}_i+\mathbf{h}$ such that $\mathbf{h}\in \mathbb{N}^d$ and $j\in \mathbb{N}$. Since $\mathbf{f}$ is a gap of $\kthick{k}(S,i)$, we have $j\leq k$ and $\mathbf{h}\in \cH(S)$. If $j<k$ then $\mathbf{f}=(j+1)\mathbf{e}_i+\mathbf{h}$ with $\mathbf{h}\in \cH(S)$ and $j+1\leq k$. Since $\mathbf{e}_i\in \kthick{k}(S,i)$, we obtain that in this case $\mathbf{f}\notin \operatorname{PF}(\kthick{k}(S,i))$, a contradiction. So, we have $j=k$. Now, take $\mathbf{s}\in S\setminus \{\mathbf{0}\}$. Since $\mathbf{f}\in \operatorname{PF}(\kthick{k}(S,i))$ then we have $\mathbf{f}+\mathbf{s}=k\mathbf{e}_i+(\mathbf{h}+\mathbf{s})\in \kthick{k}(S,i)$. By $j=k$, we obtain that $\mathbf{h}+\mathbf{s}\in S$. So, $\mathbf{h}\in \operatorname{PF}(S)$. Hence $\operatorname{PF}(\kthick{k}(S,i))\subseteq k\mathbf{e}_i+\operatorname{PF}(S)$. Now consider an element of kind $k\mathbf{e}_i+\mathbf{h}$ such that $\mathbf{h}\in \operatorname{PF}(S)$. Take $\mathbf{s}\in \kthick{k}(S,i)$. Write $\mathbf{s}=j\mathbf{e}_i+\mathbf{t}$ such that $\mathbf{t}\in \mathbb{N}^d$ and $j\in \mathbb{N}$. If $j=0$, then $\mathbf{t}\in S$ and $(k\mathbf{e}_i+\mathbf{h})+\mathbf{s}=k\mathbf{e}_i+(\mathbf{h}+\mathbf{t})\in k\mathbf{e}_i+S \subseteq \kthick{k}(S,i)$. If $j>0$, then $(k\mathbf{e}_i+\mathbf{h})+\mathbf{s}\in ((k+1)\mathbf{e}_i+\mathbb{N}^{d+1})\subseteq \kthick{k}(S,i)$. Hence, $k\mathbf{e}_i+\mathbf{h}\in \operatorname{PF}(\kthick{k}(S,i))$. 

    \noindent Now we prove (iii). We first show that $\mathfrak{C}_{\kthick{k}(S,i)}\supseteq ((k+1)\mathbf{e}_i+\mathbb{N}^{d+1})\cup (\mathfrak{C}_S+\mathbb{N}^{d+1}).$ Let $\mathbf{x}\in \mathfrak{C}_S+\mathbb{N}^{d+1}$. We can write $\mathbf{x} =\mathbf{c} + \mathbf{n} $ such that $\mathbf{c}\in \mathfrak{C}_S$ and $\mathbf{n}\in\mathbb{N}^{d+1}$. Now let $\mathbf{t}\in \mathbb{N}^{d+1}$. We can write $\mathbf{t}=j\mathbf{e}_i+\overline{\mathbf{t}}$ with $\overline{\mathbf{t}}\in \mathbb{N}^{d}$ and $j\in \mathbb{N}$. Then $\mathbf{x}+\mathbf{t}=(\mathbf{c}+\overline{\mathbf{t}})+(\mathbf{n}+j\mathbf{e}_i)\in \kthick{k}(S,i)$, since $\mathbf{c}+\overline{\mathbf{t}}\in S$. It is trivial that $((k+1)\mathbf{e}_i+\mathbb{N}^{d+1})\subseteq \mathfrak{C}_{\kthick{k}(S,i)}$. Hence, the first inclusion is proved. For the other inclusion, let $\mathbf{f}\in \mathfrak{C}_{\kthick{k}(S,i)}$. We can write $\mathbf{f}=j\mathbf{e}_i+\mathbf{y}$ with $\mathbf{y}\in \mathbb{N}^d$ and $j\in \mathbb{N}$. We can assume $j\leq k$. Otherwise, we are done. Since $\mathbf{f}\in \kthick{k}(S,i)$, we have that $\mathbf{y}\in S$. Assume $\mathbf{y}\notin \mathfrak{C}_S$. Then there exists $\mathbf{n}\in \mathbb{N}^d$ such that $\mathbf{y}+\mathbf{n}\in \cH(S)$. But this means that $\mathbf{f}+\mathbf{n}=j\mathbf{e}_i+(\mathbf{y}+\mathbf{n})\in j\mathbf{e}_i+\cH(S)\subseteq \cH(\kthick{k}(S,i))$, a contradiction. Therefore, $\mathbf{f}\in \mathfrak{C}_S+\mathbb{N}^{d+1}$. This completes the proof.
\end{proof}

\begin{example}
\label{exa:kthick}\rm 
Let $S=\mathbb{N}\setminus \{1,2,3,4,6,9,11\}=\langle 5,7,8 \rangle$. The $4$-thickening of $S$ along axis $y$ (the 2-nd) is:
	$$\kthick{4}(S,2)=\mathbb{N}^{2}\setminus \{(1,i),(2,i),(3,i),(4,i),(6,i),(9,i),(11,i)\mid i=0,1,2,3,4\}.$$	
	\noindent Put $S'=\kthick{4}(S,2)$. The $5$-thickening of $S'$ along axis $z$ (the 3-rd) is:
	$$\kthick{5}(S',3)=\mathbb{N}^{3}\setminus \{(h_1,h_2,j) \mid (h_1,h_2)\in \cH(S'), j=0,1,2,3,4,5\}.$$ 
	\begin{figure}[h!]
			\begin{tikzpicture}[scale=0.5, >=latex, font=\footnotesize]
				\usetikzlibrary{patterns}
				\draw [help lines] (0,0) grid (13,7);
				\draw [<->] (0,8) node [left] {$y$} -- (0,0)
				-- (14,0) node [below] {$x$};
				\foreach \i in {1,...,13}
				\draw (\i,1mm) -- (\i,-1mm) node [below] {$\i$}; 
				\foreach \i in {1,...,7}
				\draw (1mm,\i) -- (-1mm,\i) node [left] {$\i$}; 
				\node [below left] at (0,0) {$O$};

				\draw [mark=*] plot (1,0);
				\draw [mark=*] plot (2,0);
				\draw [mark=*] plot (3,0);
				\draw [mark=*] plot (4,0);
				\draw [mark=*] plot (6,0);
				\draw [mark=*] plot (9,0);
				\draw [mark=*] plot (11,0);
				\draw [mark=*] plot (1,1);
				\draw [mark=*] plot (2,1);
				\draw [mark=*] plot (3,1);
				\draw [mark=*] plot (4,1);
				\draw [mark=*] plot (6,1);
				\draw [mark=*] plot (9,1);
				\draw [mark=*] plot (11,1);
				\draw [mark=*] plot (1,2);
				\draw [mark=*] plot (2,2);
				\draw [mark=*] plot (3,2);
				\draw [mark=*] plot (4,2);
				\draw [mark=*] plot (6,2);
				\draw [mark=*] plot (9,2);
				\draw [mark=*] plot (11,2);
				\draw [mark=*] plot (1,3);
				\draw [mark=*] plot (2,3);
				\draw [mark=*] plot (3,3);
				\draw [mark=*] plot (4,3);
				\draw [mark=*] plot (6,3);
				\draw [mark=*] plot (9,3);
				\draw [mark=*] plot (11,3);
				\draw [mark=*] plot (1,4);
				\draw [mark=*] plot (2,4);
				\draw [mark=*] plot (3,4);
				\draw [mark=*] plot (4,4);
				\draw [mark=*] plot (6,4);
				\draw [mark=*] plot (9,4);
				\draw [mark=*] plot (11,4);
							
			\end{tikzpicture}
		\qquad
		\includegraphics[scale=0.3]{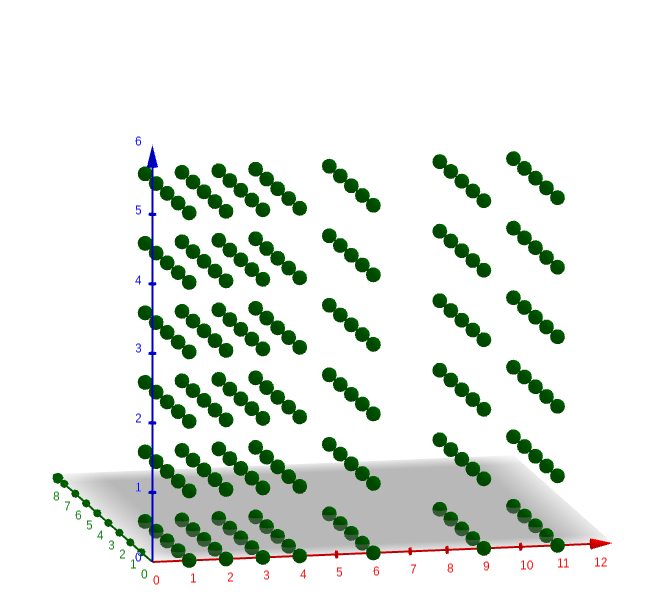}
		\caption{Pictures of the semigroups $\kthick{4}(S,2)$ and $\kthick{5}(S',3)$ of Example~\ref{exa:kthick}. The marked points are the gaps of the semigroups.}
		\label{fig:kthick}
	\end{figure}
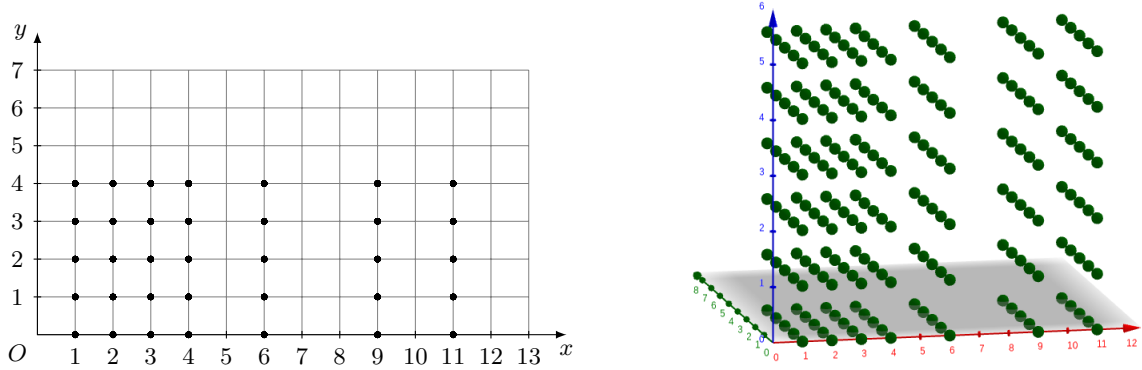   

\noindent Figure~\ref{fig:kthick} provides a graphical view of the set of gaps of the generalized numerical semigroups $S$ and $S'$. Observing that $\operatorname{PF}(S)=\{7,8,10\}$, it is not difficult to see that $\operatorname{PF}(\kthick{4}(S,2))=\{(7,4),(8,4),(10,4)\}$ and $\operatorname{PF}(\kthick{5}(S',3))=\{(7,4,5),(8,4,5),(10,4,5)\}$, as stated in Proposition~\ref{prop:kthick}.        
    \end{example}

\begin{theorem}\label{thm:thick_redtype}
    Let $\kthick{k}(S,i)$ be the $k$-thickening of a generalized numerical semigroup $S\subseteq \mathbb{N}^d$. Then $\kthick{k}(S,i)$ has maximal reduced type if and only if $S$ has maximal reduced type. 
    \end{theorem}
    \begin{proof}
    Observe first that if $\{\mathbf{n}_1,\ldots,\mathbf{n}_d\}$ is the set of minimal extremal rays of $S$, then $\{\mathbf{e}_i, \mathbf{n}_1,\ldots,\mathbf{n}_d\}$ is the set of extremal rays of $\kthick{k}(S,i)$. Assume $S$ has maximal reduced type. Let $\mathbf{f}\in \operatorname{PF}(\kthick{k}(S,i))$. So, by Proposition~\ref{prop:kthick}$,\mathbf{f}=k\mathbf{e}_i+\mathbf{f}'$ with $\mathbf{f}'\in \operatorname{PF}(S)$. We show that $\mathbf{f}+\mathbf{e}_i\in \mathfrak{C}_{\kthick{k}(S,i)}$ and $\mathbf{f}+\mathbf{n}_j\in \mathfrak{C}_{\kthick{k}(S,i)}$ for all $j\in \{1,\ldots,d\}$. In fact, $\mathbf{f}+\mathbf{e}_i=(k+1)\mathbf{e}_i+\mathbf{f}'\in (k+1)\mathbf{e}_i+\mathbb{N}^{d+1}\subseteq \mathfrak{C}_{\kthick{k}(S,i)}$, by Proposition~\ref{prop:kthick}. Moreover, if $j\in \{1\ldots,d\}$, then $\mathbf{f}+\mathbf{n_j}=k\mathbf{e}_i+(\mathbf{f}'+\mathbf{n}_j)$. Since $S$ has maximal reduced type, we have that $\mathbf{f}'+\mathbf{n}_j\in \mathfrak{C}_S$. Therefore, by Proposition~\ref{prop:kthick}, we get $\mathbf{f}+\mathbf{n}_j\in \mathfrak{C}_S+\mathbb{N}^{d+1}\subseteq \mathfrak{C}_{\kthick{k}(S,i)}$. Thus, by Proposition~\ref{prop:conductor} and Proposition~\ref{prop:maximal-redtype}, we obtain that $\kthick{k}(S,i)$ has maximal reduced type. For the converse, suppose $\kthick{k}(S,i)$ has maximal reduced type. It is sufficient to show that for all $\mathbf{f}\in \operatorname{PF}(S)$ and for all $j\in \{1,\ldots,d\}$ we have $\mathbf{f}+\mathbf{n}_j\in \mathfrak{C}_S$. So, let $\mathbf{f}\in \operatorname{PF}(S)$ and $j\in \{1,\ldots,d\}$. Assume $\mathbf{f}+\mathbf{n}_j\notin \mathfrak{C}_s$. So, there exists $\mathbf{x}\in \mathbb{N}^{d}\setminus \{\mathbf{0}\}$ such that $\mathbf{f}+\mathbf{n}_j+\mathbf{x}\in \cH(S)$. By Proposition~\ref{prop:kthick}, we have $k\mathbf{e}_i+\mathbf{f}+\mathbf{n}_j+\mathbf{x}\in \cH(\kthick{k}(S,i))$. But this is a contradiction, since $k\mathbf{e}_i+\mathbf{f}+\mathbf{n}_j\in \mathfrak{C}_{\kthick{k}(S,i)}$. This completes the proof.
    \end{proof}

    \begin{remark}\rm 
     Let $\kthick{k}(S,i)$ be the $k$-thickening of a generalized numerical semigroup $S\subseteq \mathbb{N}^d$. If $|\operatorname{PF}(S)|\geq 2$, the $\kthick{k}(S,i)$ has never minimal reduced type. In fact, if $\mathbf{f}_1,\mathbf{f}_2\in \operatorname{PF}(S)$, then $k\mathbf{e}_i+\mathbf{f}_\ell\in \operatorname{PF}(\kthick{k}(S,i))$ and $(k\mathbf{e}_i+\mathbf{f}_\ell)+\mathbf{e}_i\in \mathfrak{C}_{\kthick{k}(S,i)}$, for $\ell=1,2$. In particular, if $S$ has minimal reduced type, then $\kthick{k}(S,i)$ may not have minimal reduced type. Anyway, $\kthick{k}(S,i)$ has minimal reduced type if and only if $|\operatorname{PF}(S)|=1$.
    \end{remark}

\section{Unboundedness properties}

In this section, we study the unboundedness properties of the invariants type, reduced type, and minimal number of irreducible components of a $\mathcal{C}$-semigroup.

\subsection{Unboundedness of the type in fixed embedding dimension} 
Now, we investigate the problem of the boundedness of the type $\operatorname{t}(S)$ in the set of generalized numerical semigroups having fixed embedding dimension $e$. First, we recall what happens if $d =1$, i.e., $S$ is a numerical semigroup. If $e=2$, then $S$ is symmetric and so $\operatorname{t}(S)=1$. If $e=3$, then $\operatorname{t}(S)\leq 2$ by \cite[Corollary 10.22]{rosales2009numerical}. For $e=4$, in \cite{bresinsky} Bresinsky gave an example
of monomial curves in the affine $4$-space $\mathbb{A}^4$ defined by the integers $n_1 = 2h(2h+1)$, $n_2 = (2h-1)(2h+1)$, $n_3 = 2h(2h+1)+(2h-1)$, and $n_4 = (2h-1)2h$, where $h \geq 2$. With the help of this example, he showed that the defining ideal of a monomial curve in the affine $n$-space $\mathbb{A}^n$ does not have an upper bound on the minimal generators. Further, in \cite{stamate}, it is shown that Bresinsky's curves do not have an upper bound on the Betti numbers in terms of its embedding dimension. Since the type of a numerical semigroup corresponds to the last Betti number of the associated semigroup ring when $e=4$, one can conclude the type of a numerical semigroup does not have an upper bound with respect to the embedding dimension. 

To deduce the same property for all $e \geq 4$, we use the construction of gluing of a numerical semigroup. Let us first recall the definition of gluing of numerical semigroups.

\begin{definition} \label{NSgluing}
	Let $S_1$ and $S_2$ be two numerical semigroups minimally generated by $n_1,\ldots,n_r$ and $n_{r+1},\ldots,n_e$ respectively. Let $\mu \in S_1 \setminus \{ n_1,\ldots,n_r \}$ and $\lambda \in S_2 \setminus \{ n_{r+1}, \ldots, n_e \}$ be such that $\gcd(\lambda,\mu)=1.$ We say that $S = \langle \lambda n_1, \ldots, \lambda n_r, \mu n_{r+1}, \ldots, \mu n_{e} \rangle$ is a gluing of $S_1$ and $S_2,$ denoted by $S_1+_{\lambda \mu} S_2$.
\end{definition}

The following proposition by Nari \cite{nari} describes pseudo-Frobenius elements and the type of gluing of numerical semigroups.

\begin{proposition}[{\cite[Proposition 6.6]{nari}}]\label{pfgluing}
	If $S = S_1+_{\lambda \mu} S_2$ (as in Definition {\rm \ref{NSgluing}}), then
	\begin{enumerate}
\item[(i)]	$ \mathrm{PF}(S) = \{ \lambda f + \mu g + \lambda \mu \mid f \in \mathrm{PF}(S_1), g \in \mathrm{PF}(S_2) \}.$
	
	\item[(ii)] $\mathrm{type}(S) = \mathrm{type}(S_1)\mathrm{type}(S_2)$.
\item[(iii)] $\mathrm{F}(S) = \lambda \mathrm{F}(S_1) + \mu \mathrm{F}(S_2) + \lambda \mu$.
	\end{enumerate}
\end{proposition}

\begin{remark}\label{remarktyped=1}
{\rm Let $e \geq 5$ and $S_1$ be a numerical semigroup of embedding dimension $e-4$, and $S_2$ be a numerical semigroup generated by the integers defined in Bresinsky’s curves. Since $S_2$ has embedding dimension $4$, any gluing $S$ of $S_1$ and $S_2$ will have embedding dimension $e$. Therefore, by Proposition~\ref{pfgluing}, the type of $S$ will be the product of the types of $S_1$ and $S_2$. Since the type of $S_2$ depends on the parameter $h$, the type of $S$ will also depend on $h$. Thus, we can conclude that for fix $e \geq 4$, there is no upper bound on the set of types of numerical semigroups of a fixed embedding $e$.}
\end{remark}

\noindent Now, we consider $d > 1$. Recall that we have $\operatorname{e}(S)\geq 2d$ in this case. Let $r\geq 1$ an integer and $a\in \mathbb{N}$ such that $a+r< 2a-1$. Consider $T=\langle a, a+1,\ldots,a+r\rangle\subseteq \mathbb{N}$. Such kinds of semigroups have been studied in \cite{garcia1999numerical}. It is known that $$T=\bigcup_{k=0}^{p-1}[ka,k(a+r)]\cup [pa,+\infty),$$ where $p=\lceil \frac{a-1}{r}\rceil$. In particular, for all $k\in \{0,\ldots,p-1\}$ we have $[k(a+r)+1,(k+1)a-1]\subseteq \operatorname{H}(T)$ (that is, they are gaps of $T$). Define the following generalized numerical semigroup in $\mathbb{N}^d$, for $d\geq 2$.
$$S(a,r)=  \langle \mathbf{e}_1,\ldots, \mathbf{e}_{d-1}, a\mathbf{e}_d, (a+1)\mathbf{e}_d, \ldots, (a+r)\mathbf{e}_d, \mathbf{e}_1+\mathbf{e}_d, \mathbf{e}_2+\mathbf{e}_d,\ldots, \mathbf{e}_{d-1}+\mathbf{e}_{d}\rangle.$$

\begin{proposition} \label{prop:unbound_typeGNS}
    Let $a,r$ be positive integers with $r<a-1$. Then $\operatorname{t}(S(a,r))\geq (d-1)\left(\lceil \frac{a-1}{r}\rceil -1\right).$
\end{proposition}
\begin{proof}
For $k\in \{0,\ldots,p-1\}$ and $i\in \{1,\ldots,d-1\}$,  define the following elements: $$\mathbf{p}_{k,i}:=(a-kr-2)\mathbf{e}_i+((k+1)a-1)\mathbf{e}_d.$$
Suppose $\mathbf{p}_{k,i}\in S(a,r)$, then $\mathbf{p}_{k,i}=\sum_{j=0}^r \lambda_j (a+j)\mathbf{e}_d +\sum _{\ell=1}^{d-1} \mu_{\ell} \mathbf{e}_{\ell} +\sum _{\ell=1}^{d-1} \nu_{\ell}( \mathbf{e}_{\ell}+\mathbf{e}_d)$. Easily follows that $\mu_\ell=\nu_\ell=0$ for all $\ell\neq i$. Moreover we have $a-kr-2=\mu_i +\nu_i$ and $\sum_{j=0}^r \lambda_j (a+j)+\nu_i=(k+1)a-1$. But, as consequence, $\sum_{j=0}^r \lambda_j (a+j)=(k+1)a-1-\nu_i\geq (k+1)a-1-a+kr+2=k(a+r)+1$, that is $\sum_{j=0}^r \lambda_j (a+j)\in [k(a+r)+1,(k+1)a-1]\subseteq \operatorname{H}(T)$, a contradiction. So $\mathbf{p}_{k,i}\in \operatorname{H}(S)$ for all $k$ and $i$. We show that they are pseudo-Frobenius elements. Observing that $(k+1)a-1-(k(a+r)+1)=a-kr-2$, and if $t\in T$ then $t\mathbf{e}_d\in S$, we have:
\begin{align*}
\mathbf{p}_{k,i}+a\mathbf{e}_d  & =  (a-kr-2)\mathbf{e}_i+((k+1)a-1)\mathbf{e}_d+a\mathbf{e}_d+(k(a+r)+1)\mathbf{e}_d-(k(a+r)+1)\mathbf{e}_d\\
& =  (a-kr-2)(\mathbf{e}_i+\mathbf{e}_d)+(k(a+r)+(a+1))\mathbf{e}_d \in S(a,r).
\end{align*}
For $j\in \{1\ldots,r\}$ we have:
\begin{align*}
\mathbf{p}_{k,i}+(a+j)\mathbf{e}_d  & =  (a-kr-2)\mathbf{e}_i+((k+1)a+(a+j-1))\mathbf{e}_d \in S(a,r).
\end{align*}
Furthermore, for $\ell \in \{1,\ldots,d-1\}$:
\begin{align*}
\mathbf{p}_{k,i}+(\mathbf{e}_\ell+\mathbf{e}_d)  & =  (a-kr-2)\mathbf{e}_i+\mathbf{e}_\ell+((k+1)a)\mathbf{e}_d \in S(a,r).
\end{align*}
\begin{align*}
\mathbf{p}_{k,i}+\mathbf{e}_\ell  & =  (a-kr-2)\mathbf{e}_i+((k+1)a-1)\mathbf{e}_d+\mathbf{e}_\ell+(k(a+r)+1)\mathbf{e}_d-(k(a+r)+1)\mathbf{e}_d\\
& =  (a-kr-2)(\mathbf{e}_i+\mathbf{e}_d)+k(a+r)\mathbf{e}_d+(\mathbf{e}_\ell+\mathbf{e}_d) \in S(a,r).
\end{align*}
So $\mathbf{p}_{k,i}\in \operatorname{PF}(S(a,r))$ for all $k\in \{0,\ldots,p-1\}$ and $i\in \{1,\ldots,d-1\}$. Moreover, for $k\leq p-2$, the elements are all distinct (it happens $a-(p-1)r-2=0$ when $r=1$). Therefore we have $\operatorname{t}(S)\geq (d-1)(p-1)=(d-1)\left(\lceil \frac{a-1}{r}\rceil -1\right)$.
\end{proof}

\begin{remark} \rm
We obtained that $\operatorname{t}(S(a,r))\geq (d-1)(p-1)=(d-1)\left(\lceil \frac{a-1}{r}\rceil -1\right)$ and we have also $\operatorname{e}(S(a,r))=2d+(r-1)$. In particular, if we fix $r\geq 1$, then considering all $a>r+1$, we obtain a family of generalized numerical semigroups with fixed embedding dimension and unbounded type. Hence, for all values $e\geq 2d$, there exists a family of generalized numerical semigroups having fixed embedding dimensions equal to $e$ and unbounded type.
\end{remark}

The study of unboundedness of the type can be restricted considering special families of generalized numerical semigroups with fixed embedding dimension. For instance, in the numerical semigroup case, an open problem consists to establish if the type is bounded in the family of all almost symmetric numerical semigroups having fixed embeding dimension. This problem can be naturally extended to generalized numerical semigroup.
\begin{question}
    Consider the set of all almost symmetric generalized numerical semigroups having a fixed value of embedding dimension. Does an upper bound exist for the type of these semigroups?
\end{question}
For numerical semigroups, it is known that the type is bounded considering the embedding dimension smaller or equal than $5$ (see \cite{moscarielloedim4, moscarielloedim5}). For a generalized numerical semigroup $S\subseteq \mathbb{N}^d$, the same can be investigated starting for small values of the embedding dimension.

\subsection{Unboundedness of the reduced type in fixed embedding dimension}
The same investigation considered before for the type can be approached for the reduced type using the definition introduced in this paper. We first consider $d =1$, i.e., $S$ is a numerical semigroup. To explore the unboundedness of reduced type, we again consider Bresinky's curves defined by the integers $n_1 = 2h(2h+1)$, $n_2 = (2h-1)(2h+1)$, $n_3 = 2h(2h+1)+(2h-1)$, and $n_4 = (2h-1)2h$, where $h \geq 2$. From \cite{op3}, we describe the pseudo-Frobenius elements of the numerical semigroup generated by the integers defined in Bresinsky's curves.

\begin{theorem}\cite[Theorem 3.8]{op3} \label{pfbresinsky}
Let $S = \langle n_1,n_2,n_3,n_4 \rangle$ be a numerical semigroup generated by the integers $n_1,n_2,n_3,n_4$ defined in Bresinsky's curves. Then 
\begin{align*}
\mathrm{PF}(S) &= \left\lbrace
		\begin{array}{c}
			 (2h-1)^3 + 4h(h-2) + k(2h-1) + 1, ~\text{for}~ 0 \leq k \leq 2h-3   \\[1mm]
			 (2h-1)^3 + 4h(h-2) + 2h(2k+1) + 2, ~\text{for}~ 0 \leq k \leq 2h-2
		\end{array}
		\right\rbrace .
\end{align*}
\end{theorem}

\begin{proposition}\label{redtypebre}
  Let $S = \langle n_1,n_2,n_3,n_4 \rangle$ be a numerical semigroup generated by the integers $n_1,n_2,n_3,n_4$ defined in Bresinsky's curves. Then, the reduced type of $S$ is $h$.
\end{proposition}
\begin{proof} Let us divide the set $\mathrm{PF}(S)$ into two subsets. Consider $$\mathrm{PF}_1(S) = \{(2h-1)^3 + 4h(h-2) + k(2h-1) + 1, ~\text{for}~ 0 \leq k \leq 2h-3\},$$
and 
$$\mathrm{PF}_2(S) = \{(2h-1)^3 + 4h(h-2) + 2h(2k+1) + 2, ~\text{for}~ 0 \leq k \leq 2h-2\}.$$
 From Theorem~\ref{pfbresinsky}, note that $\mathrm{F}(S) = (2h-1)^3+4h(h-2)+2h(4h-3)+2$. Moreover, $m(S) = 2h(2h-1)$.  Therefore, any $f \in \mathrm{PF}(S) $ belongs to $\cH(S) \cap [\mathrm{F}(S)-m(S)+1, \mathrm{F}(S)]$ if and only if $f \in \mathrm{PF}_2(S)$ and $2h(2h-2) + 3 \leq 2h(2k+1)+2.$ This happens if and only if $k \geq h-1$. Therefore, $\cH(S) \cap [\mathrm{F}(S)-m(S)+1, \mathrm{F}(S)] = \{ (2h-1)^3 + 4h(h-2) + 2h(2k+1) + 2, ~\text{for}~ h-1 \leq k \leq 2h-2\}$. Note that the cardinality of this set is $h$; this completes the proof.
\end{proof}

So, in view of Proposition~\ref{redtypebre}, for $d=1$ and $\operatorname{e}(S)=4$, the reduced type does not have any upper bound. To deduce the same property for all $e \geq 4$, we explore the reduced type of a nice extension (a special kind of gluing) of a numerical semigroup. Let us recall the definition of a nice extension of a numerical semigroup.

\begin{definition}[{\cite[Definition 3.1]{arslan}}]
Let $S = \langle n_1, \ldots, n_e \rangle$ be a numerical semigroup. A numerical semigroup $S'$ is called a nice extension of $S$ if there is an element $n_{e+1} = a_1n_1 + \ldots + a_en_e \in S \setminus \{n_1, \ldots, n_e\}$ and an integer $p \leq a_1 + \ldots + a_e$ with $\mathrm{gcd}(p,n_{e+1}) = 1$ such that $S' = \langle pn_1, \ldots, pn_e, n_{e+1}\rangle$. 
\end{definition}

\begin{proposition}\label{redtypeextension}
   Let $S$ be a numerical semigroup and $S'$ be a nice extension of $S$. Then, the reduced type of $S'$ is equal to the reduced type of $S$.
\end{proposition}
\begin{proof}
Let $S' = \langle pn_1, \ldots, pn_e, n_{e+1}\rangle$ be a nice extension of $S$. Observe that, by taking $\lambda = p$ and $\mu = n_{e+1}$, we get $S' = S+_{\lambda \mu} \mathbb{N}$. Let $f \in \cH(S) \cap [\mathrm{F}(S)- \mathrm{m}(S)+1, \mathrm{F}(S)]$. This means $f \geq \mathrm{F}(S)-\mathrm{m}(S)+1.$ Therefore, we have 
$$p f + p n_{e+1} - n_{e+1} \geq p \mathrm{F}(S)- p \mathrm{m}(S)+ p + p n_{e+1} - n_{e+1}.$$
From Proposition~\ref{pfgluing}, observe that $p \mathrm{F}(S) + p n_{e+1} - n_{e+1} = \mathrm{F}(S')$ and $p f + p n_{e+1} - n_{e+1} \in \mathrm{PF}(S')$. Since the multiplicity $\mathrm{m}(S')$ of $S'$ is equal to $p\mathrm{m}(S)$, we get $p f + p n_{e+1} - n_{e+1} \geq \mathrm{F}(S') - \mathrm{m}(S') + p$. Thus, for each $f \in \cH(S) \cap [\mathrm{F}(S)- \mathrm{m}(S)+1, \mathrm{F}(S)]$ there exist $f' = p f + p n_{e+1} - n_{e+1} \in \mathrm{PF}(S')$ such that $f' \in \cH(S') \cap [\mathrm{F}(S')- \mathrm{m}(S')+1, \mathrm{F}(S')]$. Therefore, $s(S,\mathrm{m}(S')) \geq s(S,\mathrm{m}(S))$. Since the reduced type of $\mathbb{N}$ is one, the reverse inequality $s(S,\mathrm{m}(S')) \leq s(S,\mathrm{m}(S))$ follows from \cite[Theorem 3.34]{maitra2023extremal}. This completes the proof.
\end{proof}

\begin{remark}
{\rm Let $S$ be a numerical semigroup generated by the integers defined in Bresinsky's curves. Its reduced type depends on the parameter $h$. For any $e \geq 5$, we can take iterative nice extensions of $S$, and again, the reduced type of resulting numerical semigroup will depend on the parameter $h$. Thus, we can conclude that for fix $e \geq 4$, there is no upper bound on the set of reduced types of numerical semigroups of a fixed embedding $e$.}
\end{remark}

\noindent We now consider $d > 1$. We take the same class of generalized numerical semigroups introduced in the previous subsection.

\begin{proposition}\label{unboundedredtype}
    Let $a,r$ be positive integers with $r<a-1$. For all $k\in \{0,\ldots,p\}$ and $i\in \{1,\ldots,d-1\}$, let $\mathbf{p}_{k,i}$ be the pseudo Frobenius element defined in the proof of Proposition~\ref{prop:unbound_typeGNS}. Then $\mathbf{p}_{k,i}\in \operatorname{FA}(S(a,r))$.
\end{proposition}
\begin{proof}
For $0<j\leq r+1$ we have:
\begin{align*}
\mathbf{p}_{k,i}+j\mathbf{e}_d  & =  (a-kr-2)\mathbf{e}_i+((k+1)a-1)\mathbf{e}_d+j\mathbf{e}_d\\
& = (a-kr-2)\mathbf{e}_i+ka\mathbf{e}_d+(a+j-1)\mathbf{e}_d \in S(a,r). 
\end{align*}
Moreover, since $ka-r-1+j\in [ka,k(a+r)]$, for $r+1<j\leq a-1$ we have:
\begin{align*}
\mathbf{p}_{k,i}+j\mathbf{e}_d  & =  (a-kr-2)\mathbf{e}_i+((k+1)a-1)\mathbf{e}_d+j\mathbf{e}_d\\
& = (a-kr-2)\mathbf{e}_i+(ka-r+j-1)\mathbf{e}_d+(a+r)\mathbf{e}_d \in S(a,r).
\end{align*}
This means that $\mathbf{p}_{k,i}+\lambda \mathbf{e}_d\in S(a,r)$ for all $\lambda\in \mathbb{N}.$ Moreover, as shown in the proof of Proposition~\ref{prop:unbound_typeGNS}, we have $\mathbf{p}_{k,i}+\mathbf{e}_\ell\in S(a,r)$ for all $\ell \in \{1,\ldots,d-1\}$. So, we can argue $\mathbf{p}_{k,i}+\mathbf{n}\in S(a,r)$ for all $\mathbf{n}\in \mathbb{N}^d$. That is, $\mathbf{p}_{k,i}$ is a maximal gap with respect to $\leq_{\mathbb{N}^d}$, for all $k\in \{0,\ldots,p\}$ and $i\in \{1,\ldots,d-1\}$. This completes the proof by Corollary~\ref{cor:FA(S)}.
\end{proof}

\begin{remark} \rm
By Lemma~\ref{s(S,n_i)non_empty} and using the same argument explained for the type in the previous subsection, we obtained that the reduced type of $S(a,r)$ with respect to any extremal rays is unbounded with respect to embedding dimension, for all $d>1$ and all values of the embedding dimension. 
\end{remark}

\subsection{Unboundedness in irreducible decompositions for $\mathcal{C}$-semigroups}

We start with some properties about the irreducible decomposition of $\cC$-semigroups.

\begin{proposition}\label{prop:special-gaps-decomposition}
    Let $S, S_1,\ldots, S_n$ be $\cC$-semigroups such that $S \subseteq S_i$ for all $i = 1, \ldots,n$. Then the following are equivalent:

    \begin{itemize}
        \item[(i)] $S=S_1\cap S_2 \cap \cdots \cap S_n$.
        \item[(ii)] For each $\mathbf{h}\in \operatorname{SG}(S)$ there exists $i\in \{1,\ldots,n\}$ such that $\mathbf{h}\in \cH(S_i)$.

        \item[(iii)] $\cK(S_1)\cup \cdots \cup \cK(S_n)=\operatorname{SG}(S)$, where $\cK(S_i)=\{\mathbf{h}\in \operatorname{SG}(S)\mid \mathbf{h}\notin S_i\}.$ 
        
    \end{itemize}
\end{proposition}
\begin{proof}
    The proof of this proposition can be derived similar to \cite[Proposition 3.6]{cisto-irreducible}.
\end{proof}

Recall that if $S$ is irreducible, then $|SG(S)|=1$.

\begin{theorem} \label{thm:unboundeness-irreducible}
Let $S$ be $\mathcal{C}$-semigroup. If $S=S_1\cap S_2 \cap \cdots \cap S_n$ is a decomposition of $S$ such that $S_i$ is an irreducible $\cC$-semigroup for all $i\in \{1\ldots,n\}$, then $n\geq |\mathrm{Maximals}_{\leq_\cC}\cH(S)|$.
\end{theorem} 
\begin{proof}
    Let $S=S_1\cap S_2 \cap \cdots \cap S_n$ be a decomposition of $S$ into irreducible $\cC$-semigroups. We want to show that for all $\mathbf{a}\in \mathrm{Maximals}_{\leq_\cC}\cH(S)$ there exists $j\in \{1,\ldots,n\}$ such that $\operatorname{SG}(S_j)=\{\mathbf{a}\}$. Let $\mathbf{a}\in \mathrm{Maximals}_{\leq_\cC}\cH(S)$. So, $\mathbf{a}\in \operatorname{SG}(S)$ by Proposition~\ref{prop:containement}. By Proposition~\ref{prop:special-gaps-decomposition}, there exists $j\in \{1,\ldots,n\}$ such that $\mathbf{a}\in \cH(S_j)$. Since $S_j$ is irreducible, we have $\operatorname{SG}(S_j)=\{\mathbf{f}\}$ for some $\mathbf{f}\in \cH(S_j)$. Assume $\mathbf{f}\neq \mathbf{a}.$ Since $\mathbf{a}\in \cH(S_j)$ and $\operatorname{SG}(S_j)=\{\mathbf{f}\}$, by Proposition~\ref{prop:containement} we obtain $\mathbf{a}\leq_\cC \mathbf{f}$. But since $S\subseteq S_j$, we have $\mathbf{f}\in \cH(S)$ with $\mathbf{a}\leq_\cC \mathbf{f}$, a contradiction. Therefore, $\mathbf{f}=\mathbf{a}$ and $\operatorname{SG}(S_j)=\{\mathbf{a}\}$. Now, consider $\mathbf{a},\mathbf{b}\in \mathrm{Maximals}_{\leq_\cC}\cH(S)$ with $\mathbf{a}\neq \mathbf{b}$ and let $j,k\in \{1,\ldots,n\}$ such that $\operatorname{SG}(S_j)=\{\mathbf{a}\}$ and $\operatorname{SG}(S_k)=\{\mathbf{b}\}$. We can observe that $S_j\neq S_k$, because $\mathbf{b}\notin S_k$ but $\mathbf{b}\in S_j$. In fact, assume $\mathbf{b}\notin S_j$. Hence, since $\mathrm{Maximals}_{\leq_\cC}\cH(S_j)=\{\mathbf{a}\}$, we obtain $\mathbf{b}\leq_\cC \mathbf{a}$, a contradiction to $\mathbf{a},\mathbf{b}\in \mathrm{Maximals}_{\leq_\cC}\cH(S)$. This completes the proof.
\end{proof}

\begin{remark} \rm
  The bound provided in Theorem~\ref{thm:unboundeness-irreducible} can be sharp. For instance, consider the generalized numerical semigroup $S=\mathbb{N}^2\setminus \{ (0, 1), (1, 0), (1, 2), (2, 1)\}$. In this case $\cC=\mathbb{N}^2$ and $\operatorname{SG}(S)=\operatorname{FA}(S)=\{(1,2),(2,1)\}$. A decomposition of $S$ into irreducibles is $S=S_1 \cap S_2$, where $S_1=\mathbb{N}^2\setminus \{( 0, 1 ), ( 1, 0 ), ( 1, 2 )\}$ and $S_2=\mathbb{N}^2\setminus \{( 0, 1 ), ( 1, 0 ), ( 2,1 )\}$. 
\end{remark}

\noindent In \cite{bogart2024unboundedness}, it is shown that in the set of all numerical semigroups, for any positive integer $k$, there exists a numerical semigroup having a decomposition in irreducibles with at least $k$ components. One can have a similar result in the case of $\mathcal{C}$-semigroups by taking advantage of the Theorem~\ref{thm:unboundeness-irreducible}.

\begin{definition} \rm
    Let $\mathcal{C}$ be an integer cone and $A\subseteq \cC$ be a finite subset such that $\mathbf{a}\nleq_\cC \mathbf{b}$ for all $\mathbf{a},\mathbf{b}\in A$. We call a subset of this kind a $\cC$-\emph{antichain}. Denote also 
    $$\mathcal{B}_\cC(A)=\{\mathbf{x}\in \cC\setminus \{\mathbf{0}\} \mid \mathbf{x}\leq_\cC \mathbf{a}\text{ for some }\mathbf{a}\in A\}.$$ 
\end{definition}

It is not difficult to show that if $A$ is a $\cC$-antichain of an integer cone $\mathcal{C}$, then $\mathcal{C}\setminus \cB_\cC(A)$ is a $\cC$-semigroup.

\begin{proposition}
 Let $A$ be a $\cC$-antichain of an integer cone $\mathcal{C}$. Then $\mathcal{C}\setminus \cB_\cC(A)$ is a $\cC$-semigroup.   
\end{proposition}
\begin{proof}
    Let us denote $S=\mathcal{C}\setminus \cB_\cC(A)$.  It is not difficult to observe that $\cH(S)=\cB_\cC(A)$ and $\cB_\cC(A)$ is finite. So, we need only to show that it is a submonoid of $\mathbb{N}^d$. Let $\mathbf{x},\mathbf{y}\in S$ and assume $\mathbf{x}+\mathbf{y}\notin S$. Then there exists $\mathbf{a}\in A$ such that $\mathbf{x}+\mathbf{y}\leq_\cC \mathbf{a}$. In particular $\mathbf{x}\leq_\cC \mathbf{a}$, that is, $\mathbf{x}\notin S$. But this is a contradiction. So, $\mathbf{x}+\mathbf{y}\in S$ , that is, $S$ is a submonoid of $\mathbb{N}^d$.
    \end{proof}

\begin{corollary}\label{unbooundeddecomposition}
    For all $k\in \mathbb{N}$, there exists a $\cC$-semigroup $S$ such that the number of irreducible components of $S$ is at least $k$. 
\end{corollary}
\begin{proof}
 We assume that $d > 1$. For a fixed integer cone $\mathcal{C}$. Define the family of $\mathcal{C}$-semigroups:
$$\mathcal{F}=\{\cC\setminus \cB_\cC(A)\mid A\text{ is a $\cC$-antichain of }\cC\}.$$
 \noindent
 Observe that for $S \in \mathcal{F}$, $\cH(S) = \cB_\cC(A)$. Also, in this case, one can observe that $\mathrm{Maximals}_{\leq_\cC}\cH(S) = A.$ Now, it is sufficient to observe that it is always possible to find a $\cC$-antichain $A$ in $\cC$ such that $|A|\geq k$. Therefore, we can apply Theorem~\ref{thm:unboundeness-irreducible}.
\end{proof}

{\it Acknowledgement.} A major portion of this work was done during the first author's visit to the University of Messina, Italy, in May 2024. The first author would like to extend sincere thanks to the University of Messina and the group GNSAGA of Istituto Nazionale di Alta Matematica (INdAM), Italy, for their support.

	\end{document}